\documentclass[10pt,reqno,final]{amsart}
\usepackage[margin=1in]{geometry} 
\usepackage{graphicx, color, epstopdf}
\usepackage{epsfig,amssymb,amsmath,version,mathtools}
\usepackage{amssymb,version,graphicx,fancybox,mathrsfs,graphics}
\usepackage{url,hyperref}
\hypersetup{hidelinks} 
\usepackage[notcite,notref]{showkeys}
\usepackage{cases}
\usepackage{subfigure}
\usepackage{color}
\usepackage{stmaryrd}
\usepackage{multirow}
\usepackage{booktabs,siunitx}
\usepackage{booktabs}
\usepackage{multicol}
\usepackage{enumitem}

\catcode`\@=11 \theoremstyle{plain}
\@addtoreset{equation}{section}   

\@addtoreset{figure}{section}
\renewcommand\thefigure{\thesection.\@arabic\c@figure}
\@addtoreset{table}{section}
\renewcommand\thetable{\thesection.\@arabic\c@table}

\newtheorem{thm}{\bf Theorem}[section]
\newtheorem{cor}{\bf Corollary}[section]
\newtheorem{prop}{Proposition}[section]

\newtheorem{lmm}{\bf Lemma}[section]

\newenvironment{lemma}{\begin{lmm}}{\end{lmm}}
\DeclareMathOperator{\spn}{span}
\theoremstyle{remark}
\newtheorem{remark}{\bf Remark}[section]
\theoremstyle{definition}
\newtheorem{defn}{Definition}[section]

\newtheorem{exm}{\bf Example} 
\newenvironment{example}{\begin{exm}} {\end{exm}} 
\newenvironment{proofthm1}{{\noindent\emph{Proof of Theorem 3.1}}\quad}{\hfill $\square$\par}
\newenvironment{proofthm}{{\noindent\emph{Proof of Theorem 4.1}}\quad}{\hfill $\square$\par}

\graphicspath{{./Figures/} }

\begin{document}
\title[MHFs and their applications to 2D weakly singular FHIEs]{Mapped Hermite Functions and their applications to two-dimensional weakly singular Fredholm-Hammerstein integral equations }

\author[Min Wang \and Zhimin Zhang]{Min Wang$^1$
\and Zhimin Zhang$^{2}$}\thanks{$^1$Beijing Computational Science Research Center, Beijing 100193, PR China}\thanks{$^2$Department of Mathematics, Wayne State University, Detroit, Michigan 48202, USA}

\maketitle {\bf Abstract:} The Fredholm-Hammerstein integral equations (FHIEs) with weakly singular kernels exhibit multi-point singularity at the endpoints or boundaries. The dense discretized matrices result in high computational complexity when employing numerical methods. To address this, we propose a novel class of mapped Hermite functions, which are constructed by applying a mapping to Hermite polynomials.We establish fundamental approximation theory for the orthogonal functions. We propose MHFs-spectral collocation method and MHFs-smoothing transformation method to solve the two-point weakly singular FHIEs, respectively. Error analysis and numerical results demonstrate that our methods, based on the new orthogonal functions, are particularly effective for handling problems with weak singularities at two endpoints, yielding exponential convergence rate. We position this work as the first to directly study the mapped spectral method for multi-point singularity problems, to the best of our knowledge.


\bigskip
\noindent
{\bf 2000 Mathematics Subject Classification:} 
\smallskip

\noindent
{\bf Keywords}: mapped Hermite functions, mapped spectral method, Fredholm-Hammerstein integral equations, multi-point singularity, smooth transformation.

\maketitle

\section{\bf{Introduction}}
Numerous challenges in science and engineering, encompassing Laplace's equation, elasticity issues, conformal mapping, and free surface flows, among others, yield Fredholm integral equations with singular or weakly singular (in general logarithmic and algebraic) and periodic kernels. Consequently, singular or weakly singular Fredholm integral equations (FIEs) and their nonlinear counterparts have been extensively studied for decades. Fredholm-Hammerstein integral equations (FHIEs) with weakly singular kernels can be regarded as a subset of nonlinear FIEs. FHIEs emerge as reformulations of diverse physical phenomena across various fields of study including vehicular traffic, biology, economics, and more. Generally, the two-dimensional FHIE takes the form
\begin{eqnarray}\label{2D-FIE} 
 \lambda u(x,y)=g(x,y)+\int_\Omega \Theta(s,t,x,y)\psi(s,t,u(s,t))dsdt,\ \ (s,t)\in \Omega=\bar{I}\times \bar{I},
\end{eqnarray}
where $I=(0,1)$, $g(x,y)\in C(\Omega)$, $u(x,y)$ is the solution to be determined,  the function $\psi(s,t,u)$ is assumed to be nonlinear in $v$ and
\begin{eqnarray}\label{singular kernel}
\Theta(s,t,x,y)=
\begin{cases}
	|x-s|^{-\mu_1}|y-t|^{-\mu_2}k(s,t,x,y),\ \ \text{if}\ 0<\mu_1,\mu_2<1,\\
	\log |x-s|\log |y-t|k(s,t,x,y),\ \ \text{if}\ \mu_1=\mu_2=1,\\
\end{cases}
\end{eqnarray}
with $k(s,t,x,y)\in C(\Omega\times\Omega)$.

They typically have solutions that are non-smooth near the boundaries of the integration interval \cite{MR0532739,kaneko1990regularity,atkinson1997numerical}. To address the mild singularities,  a commonly employed technique is the use of graded mesh. Numerous relevant research works, including \cite{MR0640534,MR1218345,Schneider1981ProductIF,vainikko_uba_1981}, offer considerable insight into this approach. However, the application of graded mesh encounters a practical limitation wherein serious round-off errors arise as the initial step size diminishes. To mitigate this issue, numerous alternative methods have been proposed. Kaneko and Xu \cite{kaneko1996superconvergence} generalized iterated Galerkin method and iterated Galerkin-Kantorovich regularization method to approximate the solution of FHIEs of the second kind. Cao and Xu \cite{MR1974506,MR2322430} chose approximating space consisting of piecewise polynomials and of mildly singular functions that reflect the singularities of the exact solution . They developed the hybrid collocation method, in which quasi-uniform partitions were used to avoid round-off errors. Wang \cite{MR4357636} further developed hybrid multistep collocation method  for weakly singular FIEs, which converges faster with lower degrees of freedom. However, this kind of approach is that it is problem specific, which is applicable to the FHIEs with algebraic weak singular kernel. The smoothing transformation technique, which converts the existing weakly singular solution into a smoother counterpart, is also a widely adopted strategy, as evidenced by references such as \cite{MR1604395,MR4038686,MR2200102}.

The second challenge is that the discretized matrices are dense  when employing numerical methods to solve Fredholm integral equations, thereby amplifying computational complexity.  Unless the problem is relatively small, it then becomes essential to deploy fast algorithms such as the fast multipole method \cite{GREENGARD1997280}, fast direct solvers \cite{doi:10.1137/1.9781611976045}, or exponential integrators techniques \cite{MR3089380,MR4598830}. Lakestani et al. \cite{10.1016/j.cam.2011.01.043} proposed an efficient method based upon Legendre multiwavelets to reduce the solution of the given Fredholm integro-differential equation to the solution of a sparse linear system of algebraic equations. 
 The review paper \cite{greengard_gueyffier_martinsson_rokhlin_2009} stated the obvious, fast, direct methods which result in compressed or `data sparse' representations of the inverse matrix, should have a dramatic impact on simulation and design for boundary integral equations.

The spectral method has been demonstrated as one of the most efficient approaches for solving integral equations and partial differential equations, particularly when the solution exhibits sufficient smoothness. Considerable research has been conducted on the spectral method. For example, Chen et al. \cite{MR3021419,MR2552221} developed Jacobi-collocation spectral method for weakly singular Volterra integral equations with smooth solutions. Yang et al. \cite{YANG2019314} used this method to obtain the numerical solution of the weakly singular FIEs of the second kind with smooth solutions. Das et al. \cite{10.1007/s10915-015-0135-z} considered the Legendre spectral Galerkin and Legendre spectral collocation methods to approximate the solution of Hammerstein integral equation and obtained superconvergence rates for the iterated Legendre Galerkin solution for Hammerstein integral equations. 

While traditional spectral methods find wide application, they heavily rely on the regularity of the solution and can lose exponential accuracy when confronted with problems featuring solutions of limited regularity. To address this loss of accuracy stemming from solutions with low regularity, numerous nonpolynomial spectral methods have been proposed in recent years. These alternative methods are designed to effectively capture the singularities present in the solutions. Hou et al. \cite{MR3720385,hou2019muntz} proposed M\"untz spectral method and applied this nonpolynomial Jacobi spectral-collocation method to the weakly singular Volterra integral equations of the second kind  and integro-differential equations and fractional differential equations. Then Huang and Wang \cite{Huang2023Wang} extended this method to the Fredholm integral equation with algebraic weakly singular kernel by dividing the equation into two equations with Volterra type weakly singular integral operators. However, this kind of nonpolynomial spectral method is problem based and primarily tailored to address algebraic weakly singular kernels. Chen et al. \cite{MR4367670,MR4079474} introduced log orthogonal functions spectral method and obtained the spectrally accurate approximation to the problems that exhibit weakly singular behaviors at the initial time for initial value problems, or at one endpoint for boundary value problems. Drawing inspiration from the mapping idea, the objective of this paper is to develop and analysis of efficient mapped spectral methods for the Fredholm-Hammerstein integral equation. These methods are designed to effectively capture all weak singularities occurring at the boundary, rather than just focusing on single singularities near the initial time. The main strategies and contributions are highlighted as follows.\begin{itemize}
  \item We propose a novel approach called Mapped Hermite Functions (MHFs) with a tunable parameter, specifically designed to match the multi singularities present in the underlying solution. MHFs exhibit favorable properties in nonpolynomial calculus and demonstrate remarkable capability in approximating functions with singular behaviors at boundaries.
  \item We establish optimal approximation results for MHFs within appropriately weighted spaces, incorporating pseudo-derivatives. Furthermore, we design the MHFs-spectral collocation methods and analyze error estimates for the numerical scheme. Consequently, these findings demonstrate the potential for achieving true spectral accuracy in solving two-dimensional weakly singular Fredholm-Hammerstein integral equations.
  \item Additionally, we employ a smoothing transformation inspired by MHFs to the FHIEs, thereby transforming it into  new FHIEs defined on $\mathbb{R}$. Subsequently, we apply the Hermite spectral collocation method to solve the new FHIEs and derive error estimates for the numerical scheme. Through theoretical analysis and numerical experiments, we compare the spectral accuracy of the smoothing transformation numerical solution with that of the MHFs-spectral collocation solution.
\end{itemize}
While our focus remains on the prototypical FHIEs, we view this work as an initial yet crucial step towards the development of efficient spectral methods for more intricate boundary integral equations and differential equations that encompass weak singularities.

This paper is organized as follows. In section 2, we introduce the MHFs, derive optimal projection and interpolation errors in weighted pseudo-derivatives that are adapted to the involved mapping. In section 3, we apply GLOFs to solve two-dimensional FHIEs, and derive optimal error estimates. In particular, for solutions having weak singularities at $t = 0$ and $t=1$ (for 1D) or
all boundaries (for 2D), errors of the numerical solution resolved by the proposed MHFs-spectral collocation methods will converge exponentially. In section 4, we use a smooth transformation related to the MHFs to the FHIEs and obtain optimal error estimates of the solution worked out by Hermite spectral collocation method for the transformed FHIEs.  In section 5, the comparison of the numerical methods presented in section 3 and 4 is made through the numerical experiments. The effectiveness of the error estimates are demonstrated.
 
\section{\bf{Mapped Hermite Functions}}
In this section, we present the definition and properties of mapped Hermite functions (MHFs). To solve the two-endpoint weakly singular problem, throughout the paper we employ the Hermite polynomials $H_n(z)$ defined on $\mathbb{R}$, which are orthogonal with respect to the weight function $\omega(z)=e^{-z^2}$, namely,
$$
\int_{-\infty}^{+\infty} H_m(z) H_n(z) \omega(z) dz=\gamma_n \delta_{m n}, \quad \gamma_n=\sqrt{\pi} 2^n n ! .
$$
We use the mapping 
\begin{eqnarray}\label{mapping}
z(x)=\alpha\log(\frac{x}{1-x}),\ x\in I,\ \alpha>0
\end{eqnarray}
to map $I\rightarrow(-\infty,+\infty)$.
\subsection{Definition and properties}
\begin{defn}(MHFs).
For $\alpha>0$, the MHFs can be defined by
$$\mathcal{Q}^{(\alpha)}_n(x)=H_n(z(x))=H_n(\alpha\log(\frac{x}{1-x})),\quad n=0,1,\cdots.$$
\end{defn}
Next, we introduce the properties of the MHFs by the following Proposition.
\begin{prop}
The MHFs have the following properties:
\begin{enumerate}
  \item Three-term recurrence relation:
  \begin{eqnarray}\label{MHF-three-recur}
  \begin{cases}
\mathcal{Q}_{n+1}^{(\alpha)}(x)=2\alpha\log(\frac{x}{1-x})\mathcal{Q}_{n}^{(\alpha)}(x)-2n\mathcal{Q}_{n-1}^{(\alpha)}(x),\quad n\geq1,\\
\mathcal{Q}_{0}^{({(\alpha)})}(x)=1\quad\quad \mathcal{Q}_{1}^{(\alpha)}(x)=2\alpha\log(\frac{x}{1-x}).
 \end{cases}
  \end{eqnarray}
  \item Derivative relations:
  \begin{eqnarray}\label{MHF-deriv-rela1}
  \partial_x\mathcal{Q}_{n}^{(\alpha)}(x)=\lambda_n\frac{\alpha}{x(1-x)}\mathcal{Q}_{n-1}^{(\alpha)}(x),
   \end{eqnarray}
  and
    \begin{eqnarray}\label{MHF-deriv-rela2}
  \partial_x\mathcal{Q}_{n}^{(\alpha)}(x)=\frac{2\alpha^2}{x(1-x)}\log(\frac{x}{1-x})\mathcal{Q}_{n}^{(\alpha)}(x)-\frac{\alpha}{x(1-x)}\mathcal{Q}_{n+1}^{(\alpha)}(x).
    \end{eqnarray}
 where $\lambda_n$ grows linearly with respect to $n$, to be precise, $\lambda_n=2n$.
  \item Orthogonality:
  \begin{eqnarray}\label{MHF-orthogo}
 \int_0^1\mathcal{Q}_{n}^{(\alpha)}(x)\mathcal{Q}_{m}^{(\alpha)}(x)\chi^{\alpha}(x)dx=\gamma_n^{(\alpha)}\delta_{mn},
  \end{eqnarray}
  with
  \begin{eqnarray}\label{MHF-orthogo-paremeter}
  \chi^{\alpha}(x)=\frac{e^{-\alpha^2log^2(\frac{x}{1-x})}}{x(1-x)},\quad \gamma_n^{(\alpha)}=\frac{\sqrt{\pi}2^nn!}{\alpha}.
  \end{eqnarray}
  \item Sturm-Liouville problem:
   \begin{eqnarray}\label{MHF-Sturm}
  \frac{1}{\alpha^2} x(1-x)\partial_x(x(1-x)\partial_x\mathcal{Q}_{n}^{(\alpha)}(x))-2x(1-x)\log(\frac{x}{1-x})\mathcal{Q}_{n}^{(\alpha)}(x)+\lambda_n\mathcal{Q}_{n}^{(\alpha)}(x)=0.
  \end{eqnarray}
  \item MHFs-Gauss quadrature:\\
  Let $\{z_j,\omega_j\}_{j=0}^N$ be the Gauss nodes and weights of $H_{N+1}(z)$. Denote
  \begin{equation}\label{MHF-node-weight}
  x_j^{(\alpha)}=\frac{e^{z_j/\alpha}}{1+e^{z_j/\alpha}},\quad\quad \chi^\alpha_j=\frac1\alpha\omega_j.
  \end{equation}
  Then, 
  \begin{equation}\label{MHF-Gauss-quadrature}
  \int_0^1f(x)\chi^{\alpha}(x)dx=\sum\limits_{j=0}^{N}f(x^{(\alpha)}_j)\cdot\chi_j^\alpha,\quad \forall f\in\mathcal{P}_{2N+1}^{\log},
  \end{equation}
  where$$\mathcal{P}_{k}^{\log}:=\spn \{1,\log(x/(1-x)),\log^2(x/(1-x)),\cdots,\log^k(x/(1-x))\}.$$
\end{enumerate}
\end{prop}
\begin{proof}
It is easy to obtain \eqref{MHF-three-recur} by \eqref{H-three-recur} in Appendix A and the transformation \eqref{mapping}. In view of \eqref{H-deriv-rela} in Appendix A, \eqref{MHF-deriv-rela1} can be obtained by
$$\frac1\alpha x(1-x)\partial_x\mathcal{Q}_{n}^{(\alpha)}(x)=\partial_zH_n(z)=\lambda_nH_{n-1}(z)=\lambda_n\mathcal{Q}_{n-1}^{(\alpha)}(x).$$
For \eqref{MHF-deriv-rela2}, in view of \eqref{H-deriv-rela2}, we have
$$\frac1\alpha x(1-x)\partial_x\mathcal{Q}_{n}^{(\alpha)}(x)=\partial_zH_n(z)=2zH_n(z)-H_{n+1}(z)=2\alpha\log(\frac{x}{1-x})\mathcal{Q}_{n}^{(\alpha)}(x)-\mathcal{Q}_{n+1}^{(\alpha)}(x),\quad n\geq0.$$
The orthogonality of $\mathcal{Q}_{n}^{(\alpha)}(x)$ \eqref{MHF-orthogo} can be obtained by the transformation \eqref{mapping} and the orthogonality of Hermite polynomial \eqref{H-orthogo} in Appendix A:
$$\int_0^1 \mathcal{Q}_{m}^{(\alpha)}(x)\mathcal{Q}_{n}^{(\alpha)}(x)\frac{e^{-\alpha^2log^2(\frac{x}{1-x})}}{x(1-x)}dx=\frac1\alpha\int_{-\infty}^{+\infty} H_m(x) H_n(x) e^{-x^2} dx.$$
In view of $\partial_z=\frac1\alpha x(1-x)\partial_x$ and \eqref{H-Sturm}, we have
 \begin{eqnarray*}
&&\partial_z(\partial_zH_n(z))-2z\partial_zH_n(z)+\lambda_nH_n(z)\\
&=&\frac{1}{\alpha}x(1-x)\partial_x(\frac{1}{\alpha}x(1-x)\partial_x\mathcal{Q}_{n}^{(\alpha)}(x))-2x(1-x)\log(\frac{x}{1-x})\mathcal{Q}_{n}^{(\alpha)}(x)+\lambda_n\mathcal{Q}_{n}^{(\alpha)}(x)=0.
\end{eqnarray*}
Thus, \eqref{MHF-Sturm} is valid.\\
For the MHFs-Gauss quadrature, setting $x=\frac{e^{\frac z\alpha}}{1+e^{\frac z\alpha}}$, \eqref{MHF-Gauss-quadrature} can be obtained by the Hermite Gauss quadrature.
\begin{eqnarray*}
\int_0^1f(x)\chi^{\alpha}(x)dx&=&\int_{-\infty}^{\infty}f(\frac{e^{\frac z\alpha}}{1+e^{\frac z\alpha}})\chi^{\alpha}(\frac{e^{\frac z\alpha}}{1+e^{\frac z\alpha}})\frac1\alpha\frac{e^{\frac z\alpha}}{(1+e^{\frac z\alpha})^2}dz\\
&=&\int_{-\infty}^{\infty}f(\frac{e^{\frac z\alpha}}{1+e^{\frac z\alpha}})\frac1\alpha e^{-z^2}dz=\sum\limits_{j=0}^Nf(\frac{e^{\frac {z_j}{\alpha}}}{1+e^{\frac {z_j}{\alpha}}})\frac1\alpha\omega_j=\sum\limits_{j=0}^Nf(x^{(\alpha)}_j)\chi^{\alpha}_j,
\end{eqnarray*}
with $\chi^{\alpha}(x)=\frac{e^{-\alpha^2log^2(\frac{x}{1-x})}}{x(1-x)}$.
\end{proof}
We plot the distribution of MHFs-Gauss quadrature nodes $\{x_j^{(\alpha)}\}_{j=0}^{N}$ with various $N$ and $\alpha$ in Fig.\ref{Gauss-MHFs-nodes} and Fig.\ref{Gauss-MHFs-nodes-alpha0.5}. The plots demonstrate that MHFs enable us to efficiently approximate the singular function such as $x^\gamma(1-x)^\gamma$, where $0<\gamma<1$, effectively capturing two weakly singular points at $x=0$ and $x=1$.
\begin{remark}
The mappings defined by 
$$\bar{z}_1=\beta_1 atanh(2x-1),\quad x\in I,\ \beta_1>0,\quad \text{and}\quad \bar{z}_2=\beta_2 tan(\pi x-\frac{\pi}{2}),\quad x\in I,\ \beta_2>0,$$
mapping $I$ to $(-\infty,\infty)$ are other efficient transformations \cite{boyd1986polynomial}. These transformations to the Hermite polynomial may yield similar effects as the MHFs proposed in this paper.
\end{remark}
\begin{figure}[!htb]
	\centering
	\subfigure{\includegraphics[width=7cm]{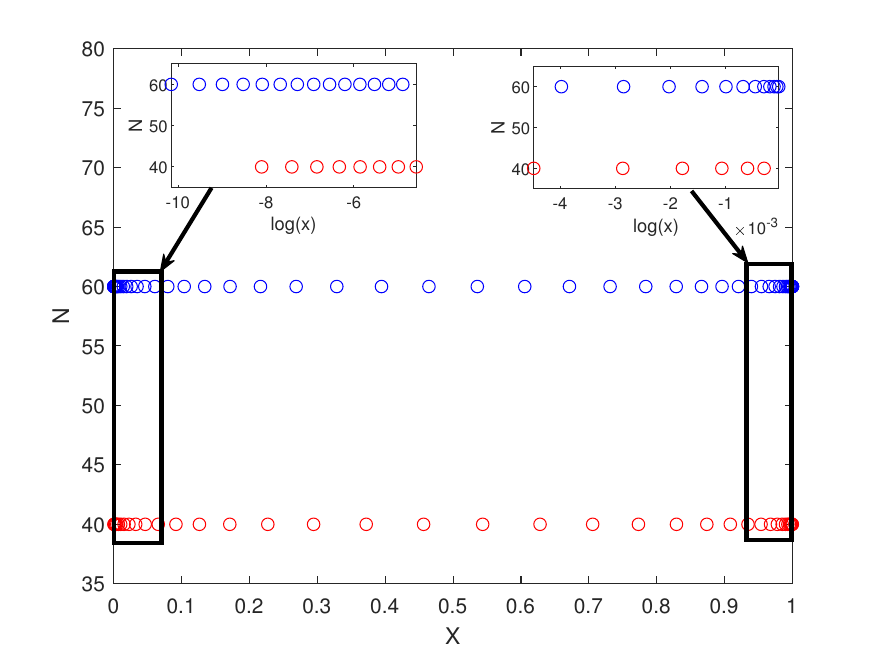}}
	\subfigure{\includegraphics[width=7cm]{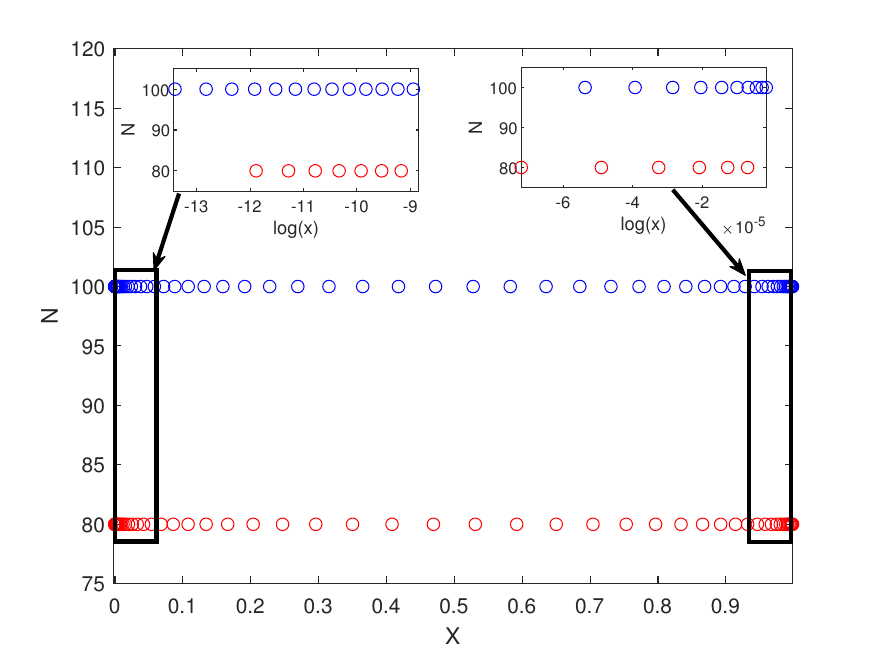}} 
	\caption{MHFs-Gauss quadrature nodes of $\mathcal{Q}_{N}^{(\alpha)}(x)$ with $N=40,60,80,100$ and $\alpha=1$  }
	\label{Gauss-MHFs-nodes}
\end{figure}
\begin{figure}[!htb]
	\centering
	\subfigure{\includegraphics[width=7cm]{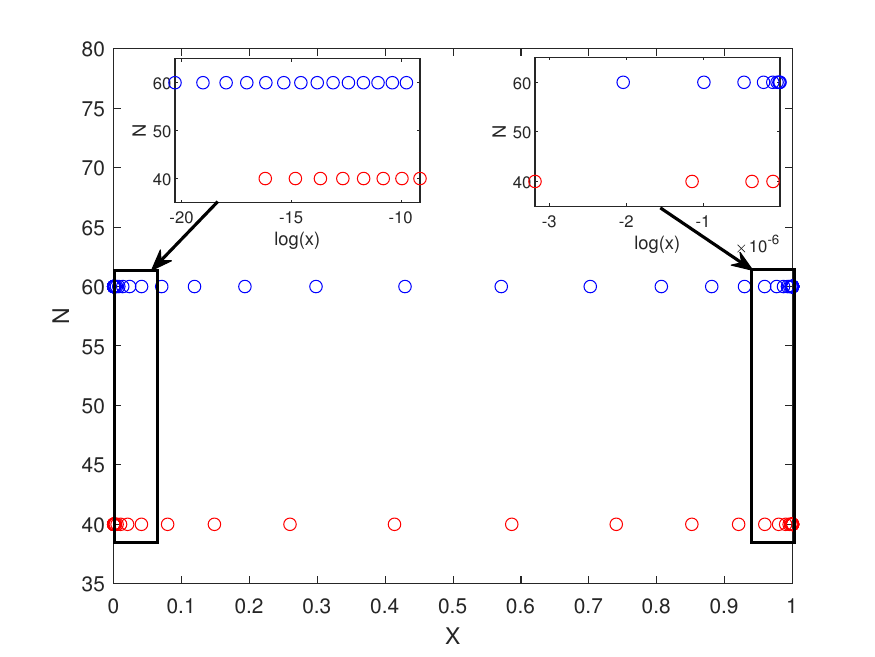}}
	\subfigure{\includegraphics[width=7cm]{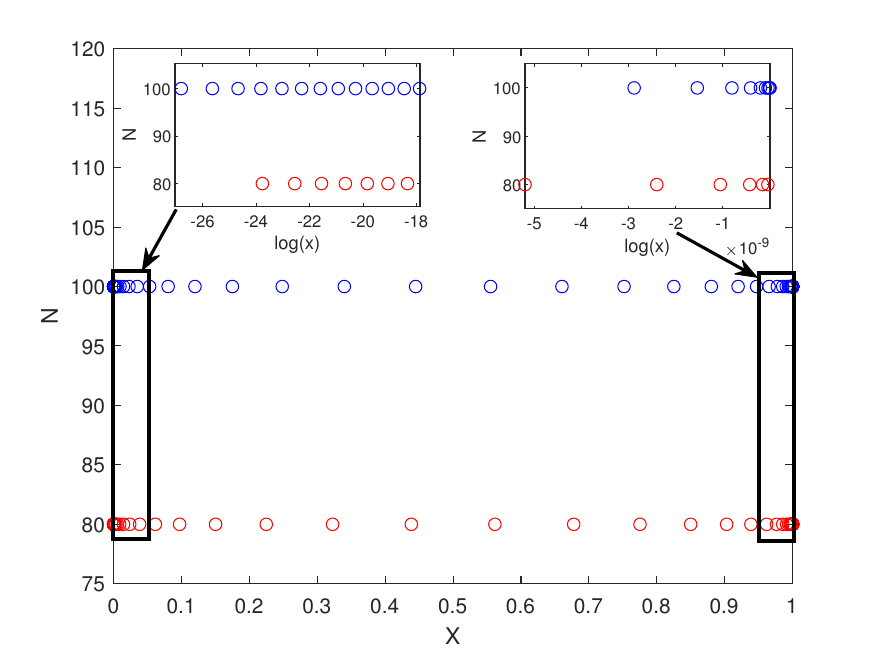}} 
	\caption{MHFs-Gauss quadrature nodes of $\mathcal{Q}_{N}^{(\alpha)}(x)$ with $N=40,60,80,100$ and $\alpha=0.5$  }
	\label{Gauss-MHFs-nodes-alpha0.5}
\end{figure}
\subsection{Projection estimate}
The projection operator $\Pi_N^{(\alpha)}:\ L^2_{\chi^{\alpha}}(I)\to \mathcal{P}_N^{\log }$ is defined by
$$(u-\Pi_N^{(\alpha)}u,v)_{\chi^{\alpha}}=0,\quad u\in L^2_{\chi^{\alpha}}(I),\ v\in\mathcal{P}_N^{\log },$$
where $\chi^{\alpha}(x)$ is the weight function defined as 
$$\chi^{\alpha}(x)=\frac{e^{-\alpha^2log^2(\frac{x}{1-x})}}{x(1-x)},\quad\alpha>0,$$
and $L^2_{\chi^{\alpha}}(I)$ is the $L^2$-weighted space defined as
$$L^2_{\chi^{\alpha}}(I)=\{u:\int_{I}|u(x)|^2\chi^{\alpha} dx<\infty\}.$$
In view of the orthogonality of the basis ${\mathcal{Q}_n^{(\alpha)}}(x)$, we have,
\begin{equation}\label{proj-operator-u}
({\Pi}_N^{(\alpha)}u)(x)=\sum\limits_{n=0}^N\widehat{u}_{n}{\mathcal{Q}_n^{(\alpha)}}(x),\quad \text{with}\ \widehat{u}_{n}=(\gamma_{n}^{(\alpha)})^{-1}\int_{I}u(x)\mathcal{Q}_n^{(\alpha)}(x)\chi^{\alpha}(x)dx.
\end{equation}
where $\gamma_{n}^{(\alpha)}$ is defined as \eqref{MHF-orthogo-paremeter}.

For any $u\in L^2_{\chi^{\alpha}}(I)$, we define the pseudo-derivative of $u$ as:
$$\widehat{\partial}_xu=x(1-x)\partial u.$$
For any arbitrary positive integer $m$,
$$\widehat{\partial}^m_{x}u=\underbrace{\widehat{\partial}_{x}\cdots\widehat{\partial}_{x}}_mu.$$
\begin{lemma}\label{k-derivative lmm}
The $k$-th order pseudo-derivatives of the MHFs are orthogonal with respect to the weight function $\chi^{\alpha}(x)$, namely,
\begin{equation}\label{k-derivative relation}
\int_0^1\widehat{\partial}^k_{x}Q_n^{(\alpha)}(x)\widehat{\partial}^k_{x}Q_m^{(\alpha)}(x)\frac{e^{-\alpha^2log^2(\frac{x}{1-x})}}{x(1-x)} dx=(\alpha^2\lambda_n\lambda_m)^{k}\gamma_{n-k}^{(\alpha+k)}\delta_{mn}.
\end{equation}
\end{lemma}
\begin{proof}
By \eqref{MHF-deriv-rela1} and \eqref{MHF-orthogo}, we can obtain the result of \eqref{k-derivative relation}.
\end{proof}
Then we define a non-uniformly weighted Sobolev space:
$$
A_{\alpha}^k(I):=\left\{v \in L_{\chi^{\alpha}}^2(I): \widehat{\partial}_x^j v \in L_{\chi^{(\alpha+j)}}^2(I), j=1,2, \cdots, k\right\}, \quad k \in \mathbb{N}.
$$
equipped with semi-norm and norm
$$
|v|_{A_{\alpha}^m}:=\left\|\widehat{\partial}_x^m v\right\|_{\chi^{\alpha+m}}, \quad\|v\|_{A_{\alpha}^m}:=\left(\sum_{k=0}^m|v|_{A_{\alpha}^k}^2\right)^{1 / 2},
$$
where the weighted $L^2$ norm $\left\|\widehat{\partial}_x^m v\right\|_{\chi^{\alpha+m}}$ is defined as
$$\left\|\widehat{\partial}_x^m v\right\|^2_{\chi^{\alpha+m}}=(\widehat{\partial}_x^m v,\widehat{\partial}_x^m v)_{\chi^{\alpha+m}}=\int_I(\widehat{\partial}_x^m v)^2\chi^{\alpha+m}(x)dx.$$
\begin{thm}\label{projection-error}
Let $m,\ N,\ k$ be positive integers and $\alpha>0$. $\forall u\in A_\alpha^m(I)$ and $0\leq k\leq\tilde{m}:=\min\{m,N+1\}$, we have
$$\left\|\widehat{\partial}_x^k\left(u-\Pi_N^{\alpha} u\right)\right\|_{\chi^{\alpha+k}} \leq(\sqrt{2}\alpha)^{k-\tilde{m}}N^{-\frac{\tilde{m}-k}{2}}\left\|\widehat{\partial}_x^{\tilde{m}}u\right\|_{\chi^{\alpha+\tilde{m}}}.$$
\end{thm}
\begin{proof}
In view of \eqref{MHF-deriv-rela1}, we have
$$\widehat{\partial}_x\mathcal{Q}_n^{(\alpha)}(x)=\frac{2n\alpha}{x(1-x)}\mathcal{Q}_{n-1}^{(\alpha)}(x).$$
Thus,
$$\widehat{\partial}^k_x\mathcal{Q}_n^{(\alpha)}(x)=\frac{(2\alpha)^kn!}{(n-k)!x^k(1-x)^k}\mathcal{Q}_{n-k}^{(\alpha)}(x).$$
$\forall u\in A_\alpha^m(I)$,
$$u(x)=\sum\limits_{n=0}^\infty\hat{u}_n^\alpha\mathcal{Q}_n^{(\alpha)}(x)$$
and
$$\|\widehat{\partial}_x^ku\|_{\chi^{\alpha+k}}^2=\sum\limits_{n=k}^\infty\frac{(2\alpha)^{2k}(n!)^2}{((n-k)!)^2x^{2k}(1-x)^{2k}}\gamma_{n-k}^{(\alpha)}|\hat{u}_n^\alpha|^2.$$
Then we have
\begin{eqnarray*}
&&\left\|\widehat{\partial}_x^k\left(u-\Pi_N^{\alpha} u\right)\right\|^2_{\chi^{\alpha+k}}=\left\|\widehat{\partial}_x^k\sum\limits_{n=N+1}^\infty\hat{u}_n^\alpha\mathcal{Q}_n^{(\alpha)}(x)\right\|^2_{\chi^{\alpha+k}}\\
&=&\sum\limits_{n=N+1}^\infty\frac{(2\alpha)^2k(n!)^2}{((n-k)!)^2x^{2k}(1-x)^{2k}}\gamma_{n-k}^{(\alpha)}|\hat{u}_n^\alpha|^2\\
&\leq&\sum\limits_{n=N+1}^\infty\frac{\frac{(2\alpha)^2k(n!)^2}{((n-k)!)^2}\gamma_{n-k}^{(\alpha)}}{\frac{(2\alpha)^2\tilde{m}(n!)^2}{((n-k)!)^2}\gamma_{n-\tilde{m}}^{(\alpha)}}\cdot\frac{(2\alpha)^2\tilde{m}(n!)^2}{((n-k)!)^2x^{2k}(1-x)^{2k}}\gamma_{n-\tilde{m}}^{(\alpha)}|\hat{u}_n^\alpha|^2\\
&\leq&(2\alpha)^{2k-2\tilde{m}}\max\limits_{N+1\leq n<\infty}\left(\frac{\Big(\frac{n!}{(n-k)!}\Big)^2\gamma_{n-k}^{(\alpha)}}{\Big(\frac{n!}{(n-\tilde{m})!}\Big)^2\gamma_{n-\tilde{m}}^{(\alpha)}}\right)\sum\limits_{n=N+1}^\infty\frac{(2\alpha)^2\tilde{m}(n!)^2}{((n-k)!)^2x^{2k}(1-x)^{2k}}\gamma_{n-\tilde{m}}^{(\alpha)}|\hat{u}_n^\alpha|^2\\
&=&(2\alpha)^{2k-2\tilde{m}}\max\limits_{N+1\leq n<\infty}\left(\frac{\Big(\frac{n!}{(n-k)!}\Big)^2\gamma_{n-k}^{(\alpha)}}{\Big(\frac{n!}{(n-\tilde{m})!}\Big)^2\gamma_{n-\tilde{m}}^{(\alpha)}}\right)\left\|\widehat{\partial}_x^{\tilde{m}}u\right\|_{\chi^{\alpha+\tilde{m}}}^2\\
&=&(2\alpha)^{2k-2\tilde{m}}\frac{(N+1-\tilde{m})!}{(N+1-k)!}\cdot2^{\tilde{m}-k}\left\|\widehat{\partial}_x^{\tilde{m}}u\right\|_{\chi^{\alpha+\tilde{m}}}^2,
\end{eqnarray*}
where the last equation is because
$$\frac{\Big(\frac{n!}{(n-k)!}\Big)^2\gamma_{n-k}^{(\alpha)}}{\Big(\frac{n!}{(n-\tilde{m})!}\Big)^2\gamma_{n-\tilde{m}}^{(\alpha)}}=\left(\frac{(n-\tilde{m})!}{(n-k)!}\right)^2\cdot\frac{2^{n-k}(n-k)!}{2^{n-\tilde{m}}(n-\tilde{m})!}=\frac{(n-\tilde{m})!}{(n-k)!}\cdot2^{\tilde{m}-k}$$
is monotonically decreasing with $n$ increasing.
Hence, we can obtain
$$\left\|\widehat{\partial}_x^k\left(u-\Pi_N^{\alpha} u\right)\right\|^2_{\chi^{\alpha+k}}\leq2^{k-\tilde{m}}\alpha^{2k-2\tilde{m}}N^{-\tilde{m}+k}\left\|\widehat{\partial}_x^{\tilde{m}}u\right\|_{\chi^{\alpha+\tilde{m}}}^2.$$
\end{proof}

In Fig.\ref{project-error}, we depict the projection errors obtained using MHFs for the function $f(x) =x^{1/2}(1-x)^{1/2}$  which features two weakly singular points at $x=0$ and $x=1$. The results are presented with a fixed degree of basis $N = 96$ and $\alpha=0.8$. We observe that the errors near the endpoints are slightly larger, resulting in the $L^\infty$-norm error not being very good, but the $L^2$-norm error is satisfactory.
 \begin{figure}[!htb]
	{\includegraphics[width=8cm]{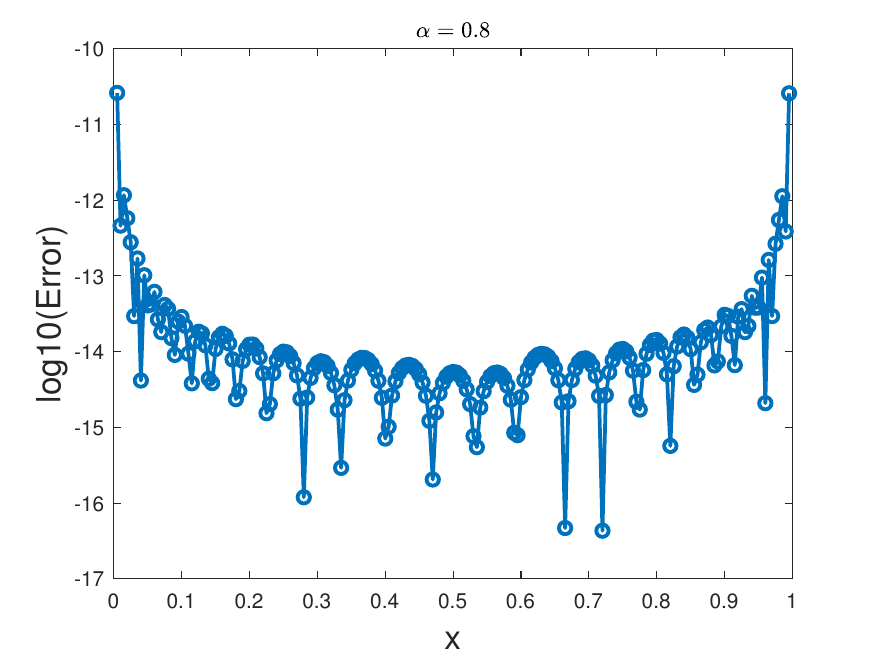}} 
	\caption{Projection error  }
	\label{project-error}
\end{figure}

For the two-dimensional variables $\boldsymbol{x}\in\Omega=I\times I$, the two-dimensional projection operator
\begin{equation*}
\boldsymbol{\Pi}_N^{\boldsymbol{\alpha}}:\ L^2_{\boldsymbol{\chi}^{\boldsymbol{\alpha}}}(\Omega)\to \boldsymbol{\mathcal{P}}_N^{\log }=\mathcal{P}_N^{\log}\circ\mathcal{P}_N^{\log}
\end{equation*}
 is defined by
\begin{equation}\label{proj-operator}
(u-\boldsymbol{\Pi}_N^{\boldsymbol{\alpha}}u,v)_{\boldsymbol{\chi}^{\boldsymbol{\alpha}}}=0,\ \ \forall u\in L^2_{\boldsymbol{\chi}^{\boldsymbol{\alpha}}}(\Omega),\ v\in \boldsymbol{\mathcal{P}}_N^{\log},
\end{equation}
where $\boldsymbol{\chi}^{\boldsymbol{\alpha}}=\chi^{\alpha_1}\cdot\chi^{\alpha_2}$ and $L^2_{\boldsymbol{\chi}^{\boldsymbol{\alpha}}}(\Omega)$ is the $L^2$-weighted space defined as
$$L^2_{\boldsymbol{\chi}^{\boldsymbol{\alpha}}}(\Omega)=\{u:\ \int_\Omega|u(\boldsymbol{x})|^2\boldsymbol{\chi}^{\boldsymbol{\alpha}}dxdy<\infty\}.$$
For any $u\in L^2_{\boldsymbol{\chi}^{\boldsymbol{\alpha}}}(\Omega)$, the two-dimensional pseudo-derivative of $u$ is defined as follows:
\begin{equation}\label{pseudo-derivative}
\widehat{D}^{\boldsymbol{k}}u=\frac{\widehat{\partial}^{|\boldsymbol{k}|}u}{\widehat{\partial}^{k_1}_{x}\widehat{\partial}^{k_2}_{y}},
\end{equation}
where $\boldsymbol{k}=\{k_1,k_2\}$, $|\boldsymbol{k}|=k_1+k_2$.

We define the two-dimensional non-uniformly weighted Sobolev space
\begin{equation*}\textbf{A}_{\boldsymbol{\alpha}}^k(\Omega)=\{v\in L_{\boldsymbol{\chi}^{\boldsymbol{\alpha}}}^2(\Omega):\widehat{D}^{\boldsymbol{j}}v\in L^2_{\boldsymbol{\chi}^{\boldsymbol{\alpha}+\boldsymbol{j}}}(\Omega),\ \boldsymbol{j}=(j_1,j_2),0\leq|\boldsymbol{j}|\leq k\},\ \ k\in\mathbb{N}\end{equation*}
equipped with norm
\begin{equation*} |v|_{\textbf{A}_{\boldsymbol{\alpha}}^m}=\left(\|\widehat{\partial}_{\frac{\beta_1-\lambda_1}{2},x}^mv\|^2_{\boldsymbol{\chi}^{\boldsymbol{\alpha}+m\boldsymbol{e}_1}}+\|\widehat{\partial}_{y}^mv\|^2_{\boldsymbol{\chi}^{\boldsymbol{\alpha}+m\boldsymbol{e}_2}}\right)^{\frac12},\quad \|v\|_{\textbf{A}_{\boldsymbol{\alpha}}^m}=\left(\sum\limits_{|\boldsymbol{k}|\leq m}\|\widehat{D}^{\boldsymbol{k}}_{\boldsymbol{\gamma}}v\|_{\boldsymbol{\chi}^{\boldsymbol{\alpha}+\boldsymbol{k}}}^2\right)^{\frac{1}{2}},
\end{equation*}
where $\boldsymbol{e}_i,\ i=1,2$ is the $i$-th unit vector of $\mathbb{R}^2$.
\begin{thm}\label{2D-projection-error}
Let $m,\ N$ be positive integers and $\alpha>0$. $\forall u\in \textbf{A}_{{\boldsymbol\alpha}}^m(\Omega)$ and $0\leq k\leq\tilde{m}:=\min\{m,N+1\}$, we have
$$\left\|\widehat{D}^k\left(u-\boldsymbol{\Pi}_N^{\alpha} u\right)\right\|_{\boldsymbol{\chi^{\alpha+k}}} \leq CN^{-\tilde{m}+k}\left\|\widehat{D}^{\tilde{m}}u\right\|_{\boldsymbol{\chi^{\alpha+\tilde{m}}}}.$$
\end{thm}
\begin{proof}
This result can be extended by Theorem \ref{projection-error}. We omit here and leave it to the reader.
\end{proof}
\subsection{Interpolation estimate}
Let $\{x_j^{(\alpha)}\}_{j=0}^N$ be the mapped Gauss points defined in \eqref{MHF-node-weight}. The generalized Lagrange function is defined as
\begin{eqnarray}\label{MHF-basic-function}
 l_j^{\alpha}(x)
=\frac{\prod\limits_{i\neq j}\Big(\log (x/(1-x))-\log(x_i^{(\alpha)}/(1-x_i^{(\alpha)}))\Big)}{\prod\limits_{i\neq j}\Big(\log(x_j^{(\alpha)}/(1-x_j^{(\alpha)}))-\log(x_i^{(\alpha)}/(1-x_i^{(\alpha)}))\Big)},\quad j=0,1,\cdots N.
\end{eqnarray}
It is clear that the functions $l_j^{\alpha}(x)$ satisfy $$l_j^{\alpha}(x_i^{(\alpha)})=\delta_{ij}.$$
The Gauss nodes defined in (\ref{MHF-node-weight}) are selected as the collocation nodes $\{x_i^{(\alpha)}\}_{i=0}^N$, the interpolation operator $$
\mathcal{I}_N^{\alpha}: C(I) \rightarrow \mathcal{P}_N^{\log},
$$ is denoted by
\begin{equation} \label{inter-operator} 
\mathcal{I}_N^{\alpha} v(x)=\sum_{j=0}^N v(x_j^{(\alpha)}) l_j^{\alpha}(x),
\end{equation}
where $l_j^{\alpha}$ are the generalized Lagrange basis functions defined in (\ref{MHF-basic-function}). It is obvious that $\mathcal{I}_N^{\alpha} v(x_j^{(\alpha)})=v(x_j^{(\alpha)}),\ j=0,1,\cdots,N$.

We will show the stability of the interpolation operator $\mathcal{I}_N^{\alpha}$. First, we give a result of the Hermite interpolation operator $\mathcal{I}_N^h$.
\begin{lmm}[\cite{aguirre2005hermite}]\label{lmm1}
There exists a constant $C>0$ such that, for all $v \in H_w^1$,
$$
\left\|\mathcal{I}_N^h v\right\|_{ w}^2 \leqslant C N^{1 / 3}\left(\|v\|_{w}^2+N^{-1}\|\varphi\|_{H_w^1}^2\right) ,
$$
where $w(x)=e^{-x^2}.$
\end{lmm}
In view of the above result, we can get the stability theorem for $\mathcal{I}_N^{\alpha}$.
\begin{thm}\label{stability}
$\forall v\in C(I)\cap A^1_\alpha(I)$, we have
$$\left\|\mathcal{I}_N^{\alpha} v\right\|^2_{\chi^{\alpha}}\leq CN^{1 / 3}\left(\alpha\|v\|_{\chi^\alpha}^2+N^{-1}\max\{1,\frac{1}{\alpha^2}\}\|v\|^2_{A_\alpha^1}\right).$$
\end{thm}
\begin{proof}
Let $x=\frac{e^{y/\alpha}}{1+e^{y/\alpha}}$ and $\tilde{v}(y)=v(x(y))$.
$$\mathcal{I}_N^{\alpha} v(x)=\mathcal{I}_N^{\alpha} v(x(y))=\mathcal{I}_N^{\alpha} \tilde{v}(y)=\sum\limits_{j=0}^N\tilde{v}(y_j)l^\alpha_j(y),\quad y\in \mathbb{R}.$$
In view of Lemma \ref{lmm1},
$$\left\|\mathcal{I}_N^{\alpha} v\right\|_{\chi^{\alpha}}^2=\left\|\mathcal{I}_N^{h} \tilde{v}(y)\right\|_{e^{-y^2}}^2\leq C N^{1 / 3}\left(\|\tilde{v}\|_{ w}^2+N^{-1}\|\tilde{v}\|_{H_w^1}^2\right).$$
For $\|\tilde{v}\|_{ w}$ and $\|\tilde{v}\|_{H_w^1}$, we have
$$\|\tilde{v}\|_{w}^2=\int_{-\infty}^{\infty}\tilde{v}^2e^{-y^2}dy=\int_0^1v^2(x)e^{-\alpha^2\log^2(\frac{x}{1-x})}\frac{\alpha}{x(1-x)}dx=\alpha\|v(x)\|_{\chi^\alpha}^2,$$
and
$$\|\tilde{v}\|_{H_w^1}^2=\int_{-\infty}^\infty\tilde{v}^2e^{-y^2}dy+\int_{-\infty}^\infty\partial_y\tilde{v}^2e^{-y^2}dy=\int_0^1\Big(v^2(x)+(\frac{1}{\alpha}\widehat{\partial}_xv(x))^2\Big)\chi^\alpha(x)dx\leq\max\{1,\frac{1}{\alpha^2}\}\|v\|^2_{A_\alpha^1}.$$
Thus,
$$\left\|\mathcal{I}_N^{\alpha} v\right\|^2_{\chi^\alpha}\leq CN^{1 / 3}\left(\alpha\|v(x)\|_{\chi^\alpha}^2+N^{-1}\max\{1,\frac{1}{\alpha^2}\}\|v\|^2_{A_\alpha^1}\right).$$
\end{proof}
Next, we present the estimate for the interpolation error.
\begin{thm}\label{interpolation-error}
Let $m$ and $N$ be positive integers, $\tilde{m}=\min\{m,N+1\}$. $\forall v\in C(I)\cap A_\alpha^m(I)$ and $\widehat\partial_xv\in A_\alpha^{m-1}(I)$, we have
$$\left\|\mathcal{I}_N^{\alpha} v-v\right\|_{\chi^\alpha}\leq C(\sqrt{2}\alpha)^{(-\tilde{m})}N^{\frac16-\frac{\tilde{m}}{2}}\left(\sqrt{\alpha}+2\max\{1,\frac{1}{\alpha}\}N^{-\frac12}+\sqrt{2}\alpha\max\{1,\frac{1}{\alpha}\}+N^{-\frac16}\right)\left\|\widehat{\partial}_x^{\tilde{m}}u\right\|_{\chi^{\alpha+\tilde{m}}}.$$
\end{thm}
\begin{proof}
By the triangle inequality, we have
$$\left\|\mathcal{I}_N^{\alpha} v-v\right\|_{\chi^\alpha}\leq\left\|\mathcal{I}_N^{\alpha} v-\Pi_N^\alpha v\right\|_{\chi^\alpha}+\left\|\Pi_N^\alpha v-v\right\|_{\chi^\alpha}.$$
For $\left\|\mathcal{I}_N^{\alpha} v-\Pi_N^\alpha v\right\|_{\chi^\alpha}$, by Theorem \ref{stability}, we have
$$\left\|\mathcal{I}_N^{\alpha} v-\Pi_N^\alpha v\right\|_{\chi^\alpha}=\left\|\mathcal{I}_N^{\alpha} \Big(v-\Pi_N^\alpha v\Big)\right\|_{\chi^\alpha}\leq CN^{1 / 6}\left(\sqrt{\alpha}\|v-\Pi_N^\alpha v\|_{\chi^\alpha}+N^{-1/2}\sqrt{\max\{1,\frac{1}{\alpha^2}\}}\|v-\Pi_N^\alpha v\|_{A_\alpha^1}\right).$$
For $\|v-\Pi_N^\alpha v\|_{A_\alpha^1}$, we have
$$\|v-\Pi_N^\alpha v\|_{A_\alpha^1}^2=|v-\Pi_N^\alpha v|_{A_\alpha^0}^2+|v-\Pi_N^\alpha v|_{A_\alpha^1}^2=2\|v-\Pi_N^\alpha v\|_{\chi^\alpha}^2+\|\widehat{\partial}_x\Big(v-\Pi_N^\alpha v\Big)\|_{\chi^{\alpha+1}}^2.$$
Thus, by the projection error in Theorem \ref{projection-error}, we can obtain
\begin{eqnarray*}
&&\left\|\mathcal{I}_N^{\alpha} v-\Pi_N^\alpha v\right\|_{\chi^\alpha}\\
&\leq& C(\sqrt{2}\alpha)^{(-\tilde{m})}N^{\frac16-\frac{\tilde{m}}{2}}\left(\sqrt{\alpha}+2\max\{1,\frac{1}{\alpha}\}N^{-\frac12}+\sqrt{2}\alpha\max\{1,\frac{1}{\alpha}\}\right)\left\|\widehat{\partial}_x^{\tilde{m}}u\right\|_{\chi^{\alpha+\tilde{m}}}.
\end{eqnarray*}
Finally,
\begin{eqnarray*}
&&\left\|\mathcal{I}_N^{\alpha} v-v\right\|_{\chi^\alpha}\\
&\leq&C(\sqrt{2}\alpha)^{(-\tilde{m})}N^{\frac16-\frac{\tilde{m}}{2}}\left(\sqrt{\alpha}+2\max\{1,\frac{1}{\alpha}\}N^{-\frac12}+\sqrt{2}\alpha\max\{1,\frac{1}{\alpha}\}+N^{-\frac16}\right)\left\|\widehat{\partial}_x^{\tilde{m}}u\right\|_{\chi^{\alpha+\tilde{m}}}.\\
\end{eqnarray*}
\end{proof}
\begin{remark}
We illustrate the derivative error of the MHFs-interpolation based on MHFs $Q^\alpha_{n}(x)$ for $f(x)=x^{0.5}(1-x)^{0.5}$ in Fig.\ref{1st-derivat-inter-error}. The superconvergence points are marked by * and they are the roots of $Q^\alpha_{n-1}(x)$ due to the derivative relation \eqref{MHF-deriv-rela1}. We can see that the errors at the superconvergence points are significantly smaller than the maximum error. Further theoretical results on superconvergence will be explored in future research.
\end{remark}

For the two-dimensional case, we denote the interpolation operator $\boldsymbol{\mathcal{I}_N^{\alpha}}:C(\Omega)\rightarrow\boldsymbol{\mathcal{P}}_N^{\log}$
\begin{equation} \label{2D-inter-operator} 
(\boldsymbol{\mathcal{I}_N^{\alpha}} v)(x,y)=(\mathcal{I}_N^{\alpha_1}\circ\mathcal{I}_N^{\alpha_2}v)(x,y)=\sum_{i=0}^N\sum_{j=0}^N v(x_i^{(\alpha_1)},y_j^{(\alpha_2)}) l_i^{\alpha_1}(x)l_j^{\alpha_2}(y),
\end{equation}
with $\boldsymbol{\alpha}=(\alpha_1,\alpha_2)$. Before giving the stability theorem of $\boldsymbol{\mathcal{I}_N^{\alpha}}$, we present the inverse inequality.
\begin{thm}\label{inver-inequ}
$\forall v\in P_N^{\log}$, for $k\geq0$, we have
$$\|\widehat{\partial}_x^kv\|_{\chi^{\alpha+k}}\leq \sqrt{C_k}N^{\frac k2}\|v\|_{\chi^\alpha},$$
with $C_k=(2\alpha^{2})^k.$
\end{thm}
\begin{proof}
$\forall v\in P_N^{\log}$, it can be expressed by
$$v(x)=\sum\limits_{i=0}^N\hat{v}_i^{\log} Q^\alpha_i(x),\quad\text{with}\ \hat{v}_i^{\log}=\frac{(v,Q_i^\alpha)_{\chi^\alpha}}{\gamma_i^{(\alpha)}},\quad \gamma_i^{(\alpha)}=\frac{\sqrt{\pi}2^ii!}{\alpha},$$
then,
\begin{equation}\label{norm-square}
\|v\|^2_{\chi^\alpha}=\sum\limits_{i=0}^N\gamma_i^{(\alpha)}|\hat{v}_i^{\log}|^2.
\end{equation}
In view of \eqref{k-derivative lmm} and \eqref{norm-square},
\begin{eqnarray*}
\|\widehat{\partial}_x^kv\|^2_{\chi^{\alpha+k}}&=&\|\widehat{\partial}_x^k\sum\limits_{i=k}^N\hat{v}_i^{\log} Q^\alpha_i(x)\|^2_{\chi^{\alpha}}=\int_0^1\left(\sum\limits_{i=k}^N\hat{v}_i^{\log} \widehat{\partial}_x^kQ^\alpha_i(x)\right)^2\chi^{\alpha}dx=\sum\limits_{i=k}^N|\hat{v}_i^{\log}|^2(2i)^{2k}\alpha^{2k}\gamma_{i-k}^{\alpha}\\
&\leq&\sum\limits_{i=k}^N|\hat{v}_i^{\log}|^2(2i)^{2k}\alpha^{2k}\gamma_i^{(\alpha)}\frac{\gamma_{i-k}^{\alpha}}{\gamma_i^{(\alpha)}}\leq\alpha^{2k}\max\limits_i\left\{(2i)^{2k}\frac{\gamma_{i-k}^{\alpha}}{\gamma_i^{(\alpha)}}\right\}\|v\|^2_{\chi^\alpha}=(2\alpha^{2})^kN^k\|v\|^2_{\chi^\alpha},
\end{eqnarray*}
where the last inequality is valid because $$(2i)^{2k}\frac{\gamma_{i-k}^{\alpha}}{\gamma_i^{(\alpha)}}=\frac{2^ki^{2k}(i-k)!}{i!}$$ are monotone increasing with $i$ increasing.
\end{proof}
Then the stability theorem of $\boldsymbol{\mathcal{I}_N^{\alpha}}$ can be derived.
\begin{thm}\label{2D-stability}
$\forall v\in C(\Omega)\cap\textbf{A}_{{\boldsymbol\alpha}}^1(\Omega)$, we have
\begin{eqnarray*}
\|\boldsymbol{\mathcal{I}_N^{\alpha}}v\|^2_{\boldsymbol{\chi^\alpha}}&\leq& CN^{\frac23}\Big(C^*+N^{-1}\max\{1,\frac{1}{\alpha^2}\}\Big)\Big(\alpha\|v\|_{\chi^\alpha}^2+N^{-1}\max\{1,\frac{1}{\alpha^2}\}\|v\|^2_{A_\alpha^1}\Big),
\end{eqnarray*}
with $C^*=\alpha+C_1\max\{1,\frac{1}{\alpha^2}\}$.
\end{thm}
\begin{proof}
In view of  Theorem \ref{stability} and Theorem \ref{inver-inequ},
\begin{eqnarray*}
\|\boldsymbol{\mathcal{I}_N^{\alpha}}v\|^2_{\boldsymbol{\chi^\alpha}}&=&\|I_N^{\alpha_1}(I_N^{\alpha_2}v)\|^2_{\boldsymbol{\chi^\alpha}}\leq CN^{\frac13}\left(\alpha\|I_N^{\alpha_2}v\|_{\chi^\alpha}^2+N^{-1}\max\{1,\frac{1}{\alpha^2}\}\|I_N^{\alpha_2}v\|^2_{A_\alpha^1}\right)\\
&\leq&CN^{\frac13}\left(\alpha\|I_N^{\alpha_2}v\|_{\chi^\alpha}^2+N^{-1}\max\{1,\frac{1}{\alpha^2}\}(\|I_N^{\alpha_2}v\|_{\chi^\alpha}^2+C_1N\|I_N^{\alpha_2}v\|_{\chi^\alpha}^2)\right)\\
&=&CN^{\frac13}\left(\Big(\alpha+C_1\max\{1,\frac{1}{\alpha^2}\}\Big)CN^{\frac13}\Big(\alpha\|v\|_{\chi^\alpha}^2+N^{-1}\max\{1,\frac{1}{\alpha^2}\}\|v\|^2_{A_\alpha^1}\Big)\right.\\
&&\left.+N^{-1}\max\{1,\frac{1}{\alpha^2}\}CN^{\frac13}\Big(\alpha\|v\|_{\chi^\alpha}^2+N^{-1}\max\{1,\frac{1}{\alpha^2}\}\|v\|^2_{A_\alpha^1}\Big)\right)\\
&=&CN^{\frac23}\left(\Big(\alpha+C_1\max\{1,\frac{1}{\alpha^2}\}\Big)\Big(\alpha\|v\|_{\chi^\alpha}^2+N^{-1}\max\{1,\frac{1}{\alpha^2}\}\|v\|^2_{A_\alpha^1}\Big)\right.\\
&&\left.+N^{-1}\max\{1,\frac{1}{\alpha^2}\}\Big(\alpha\|v\|_{\chi^\alpha}^2+N^{-1}\max\{1,\frac{1}{\alpha^2}\}\|v\|^2_{A_\alpha^1}\Big)\right)\\
&=&CN^{\frac23}\Big(\alpha+C_1\max\{1,\frac{1}{\alpha^2}\}+N^{-1}\max\{1,\frac{1}{\alpha^2}\}\Big)\Big(\alpha\|v\|_{\chi^\alpha}^2+N^{-1}\max\{1,\frac{1}{\alpha^2}\}\|v\|^2_{A_\alpha^1}\Big).
\end{eqnarray*}
\end{proof}
Then the two-dimensional interpolation error will be presented.
\begin{thm}\label{2D-interpolation-error}
Let $m$ and $N$ be positive integers, $\tilde{m}=\min\{m,N+1\}$. $\forall v\in C(\Omega)\cap \textbf{A}_{{\boldsymbol\alpha}}^m(\Omega)$ and $\widehat{D}v\in \textbf{A}_{{\boldsymbol\alpha}}^{m-1}(\Omega)$, we have
$$\left\|\boldsymbol{\mathcal{I}_N^{\alpha}} v-v\right\|_{\boldsymbol{\chi^\alpha}}\leq CN^{\frac13-\tilde{m}}\left[N^{-\frac13}+(C^*+\beta^2 N^{-1})^{1/2}\Big(\alpha^{\frac12}+\beta (N^{-\frac12}+N^{\frac12})\Big)\right]\left\|\widehat{D}^{\tilde{m}}v\right\|_{\boldsymbol{\chi^{\alpha+\tilde{m}}}},$$
with $\beta=\max\{1,\frac{1}{\alpha}\}$.
\end{thm}
\begin{proof}
The proof is similar with Theorem \ref{interpolation-error}. By Theorem \ref{2D-projection-error} and Theorem \ref{2D-stability}, we have
\begin{eqnarray*}
\left\|\boldsymbol{\mathcal{I}_N^{\alpha}} v-v\right\|_{\boldsymbol{\chi^\alpha}}&\leq&\|v-\boldsymbol{\Pi}_N^{\boldsymbol{\alpha}}v\|_{\boldsymbol{\chi^\alpha}}+\|\boldsymbol{\mathcal{I}_N^{\alpha}}(\boldsymbol{\Pi}_N^{\boldsymbol{\alpha}}v-v)\|_{\boldsymbol{\chi^\alpha}}\\
&\leq&CN^{-\tilde{m}}\left\|\widehat{D}^{\tilde{m}}v\right\|_{\boldsymbol{\chi^{\alpha+\tilde{m}}}}+CN^{\frac13}\left((C^*+N^{-1}\max\{1,\frac{1}{\alpha^2}\})^{\frac12}(\alpha^{\frac12}\|v-\boldsymbol{\Pi}_N^{\boldsymbol{\alpha}}v\|_{\chi^\alpha}\right.\\
&&\left.+N^{-\frac12}\max\{1,\frac{1}{\alpha}\}\|v-\boldsymbol{\Pi}_N^{\boldsymbol{\alpha}}v\|_{A_\alpha^1})\right)\\
&\leq&CN^{-\tilde{m}}\left\|\widehat{D}^{\tilde{m}}v\right\|_{\boldsymbol{\chi^{\alpha+\tilde{m}}}}+CN^{\frac13-\tilde{m}}(C^*+N^{-1}\max\{1,\frac{1}{\alpha^2}\})^{\frac12}\\
&&\cdot\Big(\alpha^{\frac12}+(N^{-\frac12}+N^{\frac12})\max\{1,\frac{1}{\alpha}\}\Big) \left\|\widehat{D}^{\tilde{m}}v\right\|_{\boldsymbol{\chi^{\alpha+\tilde{m}}}}\\
&\leq&CN^{\frac13-\tilde{m}}\left[N^{-\frac13}+(C^*+N^{-1}\max\{1,\frac{1}{\alpha^2}\})^{\frac12}\Big(\alpha^{\frac12}+(N^{-\frac12}+N^{\frac12})\max\{1,\frac{1}{\alpha}\}\Big)\right]\left\|\widehat{D}^{\tilde{m}}v\right\|_{\boldsymbol{\chi^{\alpha+\tilde{m}}}}.
\end{eqnarray*}
\end{proof}
Next, we will show the estimate for the MHFs-Gauss quadrature error.
\begin{thm}\label{quadra-thm}
Let $ \{x_j^{(\alpha)}\}_{j=0}^N$ and $\{\chi^\alpha_j\}_{j=0}^N$ be defined in \eqref{MHF-node-weight}. Let $\tilde{m}$ be defined as Theorem \ref{interpolation-error}, $N$ be a positive integer and $\alpha>0$. $\chi^\alpha(x)$ is the weight function defined in \eqref{MHF-orthogo-paremeter}. $\forall v\in C(I)\cap A^m_\alpha(I)$, we have 
$$\Big|\int_0^1v(x)\chi^\alpha(x)dx-\sum\limits_{j=0}^Nv(x_j^{(\alpha)})\chi^\alpha_j\Big|\leq \hat{C}N^{\frac16-\frac{\tilde{m}}{2}}\left(\sqrt{\alpha}+2\max\{1,\frac{1}{\alpha}\}N^{-\frac12}+\sqrt{2}\alpha\max\{1,\frac{1}{\alpha}\}+N^{-\frac16}\right)\left\|\widehat{\partial}_x^{\tilde{m}}u\right\|_{\chi^{\alpha+\tilde{m}}},
$$
where $\hat{C}=C\frac{\pi^{1/4}}{\alpha^{1/2}}(\sqrt{2}\alpha)^{(-\tilde{m})}$.
\end{thm}
\begin{proof}
In view of 
$$\sum\limits_{j=0}^Nv(x_j^{(\alpha)})\chi^\alpha_j=\int_0^1\mathcal{I}_N^{\alpha} v(x)\chi(x)dx,$$
and
$$\int_0^1\chi(x)dx=\int_0^1\frac{e^{-\alpha^2log^2(\frac{x}{1-x})}}{x(1-x)}dx=\frac1\alpha\int_{-\infty}^{\infty}e^{-x^2}dx=\frac1\alpha\sqrt{\pi},$$
we have
\begin{eqnarray*}
\Big|\int_0^1v(x)\chi^\alpha(x)dx-\sum\limits_{j=0}^Nv(x_j^{(\alpha)})\chi^\alpha_j\Big|^2&=&\Big|\int_0^1v(x)\chi^\alpha(x)dx-\int_0^1\mathcal{I}_N^{\alpha} v(x)\chi(x)dx\Big|^2\\
&=&\Big|\int_0^1(\mathcal{I}_N^{\alpha} v(x)-v(x))\chi^\alpha(x)dx\Big|^2\\
&\leq&\frac{\sqrt{\pi}}{\alpha}\left\|\mathcal{I}_N^{\alpha} v-v\right\|^2_{\chi^\alpha}.
\end{eqnarray*}
Finally, combining Theorem \ref{interpolation-error} leads to the result.
\end{proof}
The error estimate of the two-dimensional MHFs-Gauss quadrature error can be similarly presented.
\begin{thm}\label{2D-quadra-thm}
Let $ \{x_i^{(\alpha_1)}\}_{i=0}^N,\ \{y_j^{(\alpha_2)}\}_{j=0}^N$ and $\{\chi^{\alpha_1}_i\}_{i=0}^N,\ \{\chi^{\alpha_2}_j\}_{j=0}^N$ be defined in \eqref{MHF-node-weight}. Let $\tilde{m}$ be defined as Theorem \ref{interpolation-error}, $N$ be a positive integer and $\alpha>0$. $\chi^{\alpha_1}(x)$ and $\chi^{\alpha_2}(y)$ are the weight functions defined in \eqref{MHF-orthogo-paremeter}. Let $\boldsymbol{\chi^{\alpha}}=\chi^{\alpha_1}(x)\chi^{\alpha_2}(y)$. $\forall v\in C(\Omega)\cap \textbf{A}_{{\boldsymbol\alpha}}^m(\Omega)$, we have 
\begin{eqnarray*}
   && \Big|\int_\Omega v(x,y)\boldsymbol{\chi^{\alpha}}dxdy-\sum\limits_{i=0}^N\sum\limits_{j=0}^Nv(x_j^{(\alpha)},y_j^{(\alpha_2)})\chi^{\alpha_1}_i\chi^{\alpha_2}_j\Big|\\
  &\leq& CN^{\frac13-\tilde{m}}\left[N^{-\frac13}+(C^*+\beta^2 N^{-1})^{1/2}(\alpha^{\frac12}+\beta (N^{-\frac12}+N^{\frac12}))\right]\left\|\widehat{D}^{\tilde{m}}u\right\|_{\boldsymbol{\chi^{\alpha+\tilde{m}}}},
\end{eqnarray*}
with $\beta=\max\{1,\frac{1}{\alpha}\}$.
\end{thm}
\begin{proof}
By the similar deducing with Theorem \ref{quadra-thm}, we can obtain
$$\Big|\int_\Omega v(x,y)\boldsymbol{\chi^{\alpha}}dxdy-\sum\limits_{i=0}^N\sum\limits_{j=0}^Nv(x_j^{(\alpha)},y_j^{(\alpha_2)})\chi^{\alpha_1}_i\chi^{\alpha_2}_j\Big|\leq C \left\|\boldsymbol{\mathcal{I}_N^{\alpha}} v-v\right\|_{\boldsymbol{\chi^\alpha}}.$$
\end{proof}
\subsection{Test of the MHFs-Gauss quadrature accuracy }
We consider the following example to verify the accuracy of the MHFs-quadrature
$$\int_0^1f(x)(-\log(x(1-x)))dx,$$
with $f(x)=x^{1/2}$ and $\log(x)$, respectively. The quadrature errors are shown in Fig.\ref{Gauss-MHFs-quadrature}.
\begin{figure}[!htb]
	\begin{minipage}[t]{0.48\textwidth}
    \includegraphics[width=7cm]{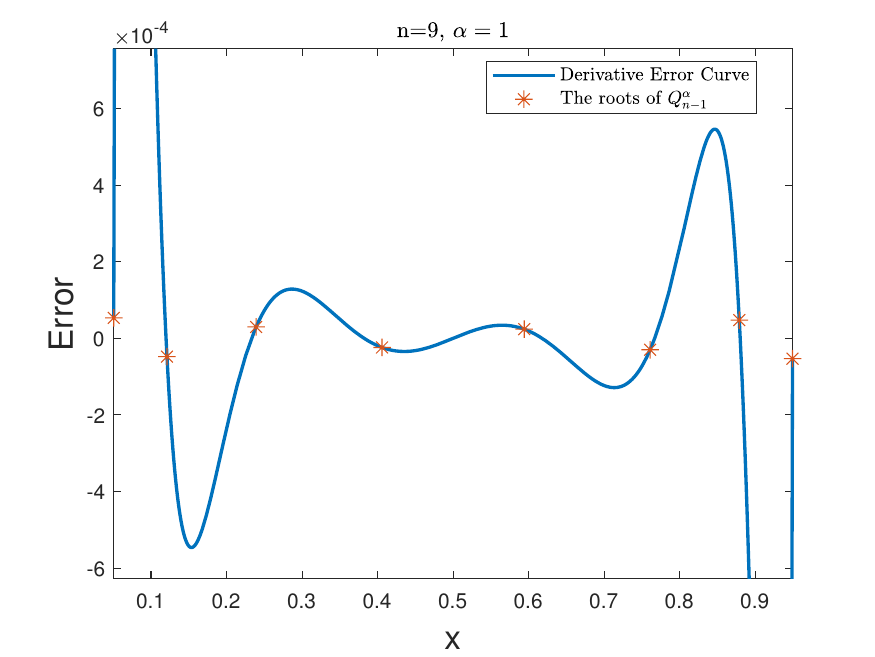} 
	\caption{\small{Derivative interpolation error } }\label{1st-derivat-inter-error}
\end{minipage}
	\begin{minipage}[t]{0.48\textwidth}
    \includegraphics[width=7cm]{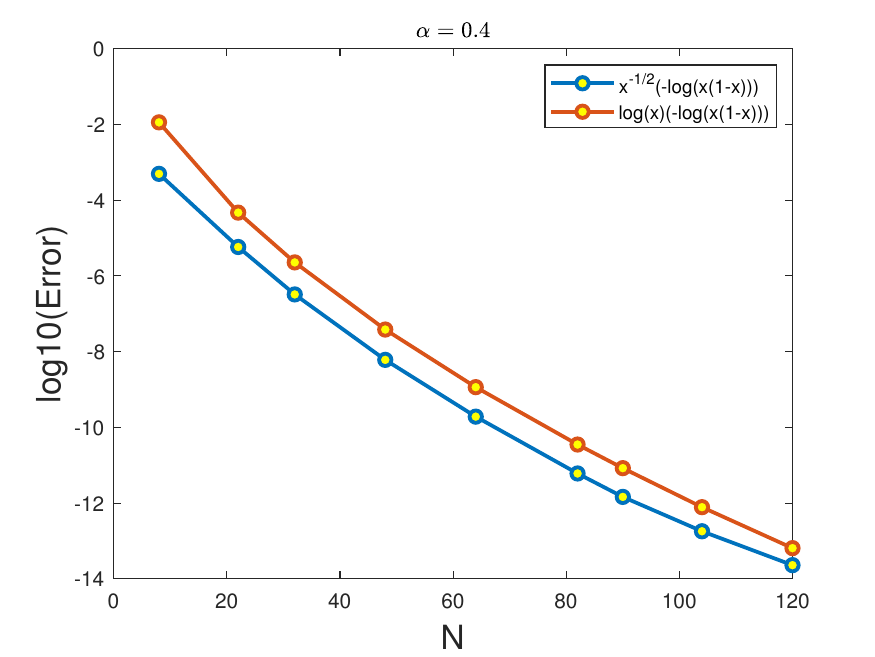} 
	\caption{\small {MHFs Gauss quadrature error  }}\label{Gauss-MHFs-quadrature}
\end{minipage}
\end{figure}
\begin{remark}
The parameter $\alpha$ of MHFs can be varied in the numerical simulations. Through the theoretical analysis and numerical computation, the MHFs numerical schemes with different $\alpha$ values can achieve exponential convergence rates. However, the constant parameter in the convergence result will influence the rate. In the future, the optimal selection of $\alpha$ will be determined by more refined analysis.
\end{remark}

\section{\bf{MHFs-spectral collocation method for weakly singular Fredholm-Hammerstein integral equations}}
In this section, we will present the MHFs-spectral collocation method for two-dimensional weakly singular Fredholm integral equations with singularities at both sides of the interval. The numerical scheme and the error analysis will be deduced in detail.

First, the operator form of \eqref{2D-FIE} is as follows
\begin{equation}\label{FIE-operator}
\lambda u(x,y)=g(x,y)+\left(\boldsymbol{\mathcal{K}}\boldsymbol{\mathcal{R}}u\right)\left(x,y\right),\ \left(x,y\right)\in \Omega,
\end{equation}
where $\boldsymbol{\mathcal{K}}\boldsymbol{\mathcal{R}}(\cdot)$ is the weakly singular integral operator defined by
$$\boldsymbol{\mathcal{K}}u=\int_\Omega\Theta(s,t,x,y)u(s,t)dsdt,$$ and $$\boldsymbol{\mathcal{R}}u=\psi(s,t,u(s,t)).$$
To see that the integral operator $\boldsymbol{\mathcal{K}}\boldsymbol{\mathcal{R}}$ with kernels $\log|x-s|\log|y-t|$ and $|x-s|^{-\mu_1}|y-t|^{-\mu_2}$ are not smooth in the manner that is true with differentiable kernel functions, let $u_0\equiv1$, $k=1$ and $\psi(s,t,x,y,u(s,t))=u(s,t)$ on $\Omega$, and calculate $\boldsymbol{\mathcal{K}}\boldsymbol{\mathcal{R}}u_0$:
\begin{eqnarray}\label{algebraic-weak}
\nonumber&&\boldsymbol{\mathcal{K}}\boldsymbol{\mathcal{R}}u_0(x,y)=\int_\Omega|x-s|^{-\mu_1}|y-t|^{-\mu_2}dsdt\\
&=&\frac{1}{(1-\mu_1)(1-\mu_2)}(x^{1-\mu_1}+(1-x)^{1-\mu_1})(y^{1-\mu_2}+(1-y)^{1-\mu_2}),\quad 0<\mu_1,\mu_2<1,
\end{eqnarray}
\begin{eqnarray}\label{log-weak}
\nonumber&&\boldsymbol{\mathcal{K}}\boldsymbol{\mathcal{R}}u_0(x,y)=\int_\Omega\log |x-s|\log |y-t|dsdt\\
&=&(x\log x+(1-x)\log (1-x)-1)(y\log y+(1-y)\log (1-y)-1),\quad \mu_1=\mu_2=1.
\end{eqnarray}
The function $\boldsymbol{\mathcal{K}}\boldsymbol{\mathcal{R}}u_0$ is not continuously differentiable on $\Omega$, whereas $u_0$ is $C^\infty$ function.

Throughout the article, the following assumptions are made on $g,\ k$ and $\psi(\cdot, \cdot, u()$) :\\
A1. $g \in \mathcal{C}^m(\Omega),\ m\geq1$.\\
A2. $k\in \mathcal{C}^1(\Omega \times \Omega),\ M=\sup\limits _{t, s,x,y \in I}|k(s,t,x,y)|$.\\
A3. $B=\sup\limits _{s \in I}\left|\psi^{(0,0,1)}(s,t, u(s))\right|$, with $\psi^{(0,0,1)}(s,t, u(s))=\frac{\partial}{\partial u}\psi(s,t, u(s))$ .\\
A4. The nonlinear function $\psi(s, t,u)$ is bounded and continuous over $\Omega \times \mathbb{R}$. $\psi(s, t,u)$ is Lipschitz continuous in $u$, that is, for any $u_1, u_2 \in$ $\mathbb{R}, \exists c_1>0$ such that
$$
\left|\psi\left(s,t, u_1\right)-\psi\left(s,t, u_2\right)\right| \leq c_1\left|u_1-u_2\right| .
$$
A5.  The partial derivative $\psi^{(0,0,1)}(s,t, u(s))$ of $\psi$ w.r.t. the second variables exists and is Lipschitz continuous in $u$ that is, for any $u_1, u_2 \in$ $\mathbb{R}, \exists c_2>0$ such that
$$
\left|\psi^{(0,0,1)}\left(s, t, u_1\right)-\psi^{(0,0,1)}\left(s,t, u_2\right)\right| \leq c_2\left|u_1-u_2\right| .
$$
Let
$$h(s,t,x,y)=\begin{cases}
	|x-s|^{-\mu_1}|y-t|^{-\mu_2},\ \ \text{if}\ 0<\mu_1,\mu_2<1,\\
	\log |x-s|\log |y-t|,\ \ \text{if}\ \mu_1=\mu_2=1.\\
\end{cases}$$
From \eqref{algebraic-weak} and \eqref{log-weak}, we obtain
$$\int_\Omega h(s,t,x,y)dsdt<P<\infty.$$
We assume that the constants $M,\ P$ and $c_1$ satisfy the condition that $MPc_1<1$. Then according to the analysis of \cite{kaneko1990regularity} and from assumptions (2) and (4), it follows that the Eq.\eqref{2D-FIE}, has a unique solution.
\subsection{The numerical scheme}
The MHFs-spectral collocation method for \eqref{2D-FIE} is to find $u_{\boldsymbol{N}}^{\boldsymbol{\log}}\in \boldsymbol{\mathcal{P}_{N}^{\log}}=\mathcal{P}_{N}^{\log}\circ\mathcal{P}_{N}^{\log}$ such that, for $i,j=0,1,\cdots,N$
\begin{equation}\label{FIE_nonlinear_system}
\lambda u_{\boldsymbol{N}}^{\boldsymbol{\log}}\left(x_i^{(\alpha_1)},y_j^{(\alpha_2)}\right)=g\left(x_i^{(\alpha_1)},y_j^{(\alpha_2)}\right)+(\boldsymbol{\mathcal{K}}\boldsymbol{\mathcal{R}}u_{\boldsymbol{N}}^{\boldsymbol{\log}})\left(x_i^{(\alpha_1)},y_j^{(\alpha_2)}\right),
\end{equation}
where $ \left(x_i^{(\alpha_1)},y_j^{(\alpha_2)}\right)$ are the collocation points defined in \eqref{MHF-node-weight}.
$$u_{\boldsymbol{N}}^{\boldsymbol{\log}}(x,y)=\sum\limits_{i=0}^N\sum\limits_{j=0}^NU_{ij}l_i^{\alpha_1}(x) l_j^{\alpha_2}(y),$$
where $l_i^{\alpha_k}(x),\ k=1,2$ is the generalized Lagrange basis function defined in \eqref{MHF-basic-function}.

And $\boldsymbol{\mathcal{K}}\boldsymbol{\mathcal{R}}u_{\boldsymbol{N}}^{\boldsymbol{\log}}$ can be approximated by MHFs-Gauss quadrature $(\boldsymbol{\mathcal{K}}\boldsymbol{\mathcal{R}})_Nu_{\boldsymbol{N}}^{\boldsymbol{\log}}$:
\begin{eqnarray}\label{2D-quadra}
\nonumber(\boldsymbol{\mathcal{K}}\boldsymbol{\mathcal{R}})_Nu_{\boldsymbol{N}}^{\boldsymbol{\log}}&=&\sum\limits_{k=0}^{NI}\sum\limits_{l=0}^{NI}\frac{|x-s_k^{(\alpha_1)}|^{-\mu_1}|y-t_l^{(\alpha_2)}|^{-\mu_2}}{\chi^{\alpha_1}(s_k^{(\alpha_1)})\chi^{\alpha_2}(t_l^{(\alpha_2)})}k(s_k^{(\alpha_1)},t_l^{(\alpha_2)},x,y)\\
\nonumber&&\cdot\psi(s_k^{(\alpha_1)},t_l^{(\alpha_2)},x,y,u_{\boldsymbol{N}}^{\boldsymbol{\log}}(s_k^{(\alpha_1)},t_l^{(\alpha_2)}))\chi^{\alpha_1}_k\chi^{\alpha_2}_l,\quad\nonumber0<\mu_1,\mu_2<1,\\
\nonumber(\boldsymbol{\mathcal{K}}\boldsymbol{\mathcal{R}})_Nu_{\boldsymbol{N}}^{\boldsymbol{\log}}&=&\sum\limits_{k=0}^{NI}\sum\limits_{l=0}^{NI}\frac{\log|x-s_k^{(\alpha_1)}|\log|y-t_l^{(\alpha_2)}|}{\chi^{\alpha_1}(s_k^{(\alpha_1)})\chi^{\alpha_2}(t_l^{(\alpha_2)})}k(s_k^{(\alpha_1)},t_l^{(\alpha_2)},x,y)\\
&&\cdot\psi(s_k^{(\alpha_1)},t_l^{(\alpha_2)},x,y,u_{\boldsymbol{N}}^{\boldsymbol{\log}}(s_k^{(\alpha_1)},t_l^{(\alpha_1)}))\chi^{\alpha_1}_k\chi^{\alpha_2}_l,\quad\mu_1=\mu_2=1,
\end{eqnarray}
where $\chi^{\alpha_1}(x)=\frac{e^{-{\alpha_1}^2log^2(\frac{x}{1-x})}}{x(1-x)}$, $NI$ is a positive integer. 

Thus, the discretized MHFs-spectral collocation method for \eqref{2D-FIE} is to find $\hat{u}_{\boldsymbol{N}}^{\boldsymbol{\log}}\in \boldsymbol{\mathcal{P}_{N}^{\log}}=\mathcal{P}_{N}^{\log}\circ\mathcal{P}_{N}^{\log}$ such that, for $i,j=0,1,\cdots,N$
\begin{equation}\label{dis-FIE_nonlinear_system}
\lambda \hat{u}_{\boldsymbol{N}}^{\boldsymbol{\log}}\left(x_i^{(\alpha_1)},y_j^{(\alpha_2)}\right)=g\left(x_i^{(\alpha_1)},y_j^{(\alpha_2)}\right)+\left((\boldsymbol{\mathcal{K}}\boldsymbol{\mathcal{R}})_N\hat{u}_{\boldsymbol{N}}^{\boldsymbol{\log}}\right)\left(x_i^{(\alpha_1)},y_j^{(\alpha_2)}\right),
\end{equation}
with$$\hat{u}_{\boldsymbol{N}}^{\boldsymbol{\log}}=\sum\limits_{i=0}^N\sum\limits_{j=0}^N\hat{U}_{ij}l_i^{\alpha_1}(x) l_j^{\alpha_2}(y).$$

\subsection{Error estimate}

In this subsection, we present the two-dimensional convergence analysis of the numerical scheme \eqref{dis-FIE_nonlinear_system}.
\begin{thm}\label{2D-MHFs-error-thm}
Let $u$ be the exact solution of Eq.\eqref{2D-FIE} and $W_{N i}$ is defined in \eqref{2D-quadra-W}. Let $\tilde{m}=\min\{m,N+1\}$. Suppose that the forcing function $g \in C^m( \Omega)$. Then, there exists a positive integer $N_0$ such that for all $N>N_0$, Eq.\eqref{2D-FIE} has a unique solution $\hat{u}^{\boldsymbol{\log}}_{\boldsymbol{N}}\in \boldsymbol{\mathcal{P}_{N}^{\log}}$ and there exists a positive constant $C$ independent of $N$ such that
$$\|u-\hat{u}^{\boldsymbol{\log}}_{\boldsymbol{N}}\|_{\boldsymbol{\chi^\alpha}}\leq C N^{\frac56-\tilde{m}}\left\|\widehat{D}^{\tilde{m}}u\right\|_{\boldsymbol{\chi^{\alpha+\tilde{m}}}}.$$
\end{thm}

To prevent distracting from the main results, we postpone the proof to Appendix B.

\begin{remark}\label{1D-MHFs-error-thm}
We consider one-dimensional weakly singular Fredholm-Hammerstein integral equation
%
\begin{equation}\label{1D-FIE-operator}
\lambda u(x)=g(x)+\left(\mathcal{K}\mathcal{R}u\right)\left(x\right),\ x\in I,
\end{equation}
where $\mathcal{KR}(\cdot)$ is the weakly singular integral operator defined by
$$(\mathcal{K}u)=\int_0^1\theta(s,x)u(s)ds,\quad (\mathcal{R}u)=\psi(s,x,u(s)),$$
with
\begin{eqnarray}\label{1D-singular kernel}
\theta(s,x)=
\begin{cases}
	|x-s|^{-\mu_1}k(s,x),\ \ \text{if}\ 0<\mu_1<1,\\
	\log |x-s|k(s,x),\ \ \text{if}\ \mu_1=\mu_2=1,\\
\end{cases}
\end{eqnarray}
with $k(s,x)\in C(I)$.

The MHFs-spectral collocation method for \eqref{1D-FIE-operator} is to find $\hat{u}_N^{\log}\in\mathcal{P}_{N}^{\log}$ such that
\begin{equation}\label{1D-FIE-nonlinear-system}
\lambda \hat{u}_N^{\log}(x_i^{(\alpha)})=g(x_i^{(\alpha)})+\left((\mathcal{KR})_N\hat{u}_N^{\log}\right)(x_i^{(\alpha)}),\quad i=0,1,\cdots N,
\end{equation}
with $x_i^{(\alpha)}$ defined in \eqref{Gauss-MHFs-nodes} and
$$\hat{u}_N^{\log}(x)=\sum\limits_{i=0}^N\hat{U}_il_i^\alpha(x),$$
where $l_i^\alpha(x)$ is the generalized Lagrange basis function defined in \eqref{MHF-basic-function}. $(\mathcal{KR})_N$ is the MHFs-Gauss quadrature operator defined as 
\begin{eqnarray*}
&&(\mathcal{KR})_N\hat{u}_{N}^{\log}=\sum\limits_{k=0}^{NI}|x-s_k|^{-\mu}k(s_k,x)\frac{1}{\chi^\alpha(s_k)}\psi(s_k,\hat{u}_{N}^{\log}(s_k))\chi^\alpha_k,\quad0<\mu<1,\\
&&(\mathcal{KR})_N\hat{u}_{N}^{\log}=\sum\limits_{k=0}^{NI}\log|x-s_k|k(s_k,x)\frac{1}{\chi^\alpha(s_k)}\psi(s_k,\hat{u}_{N}^{\log}(s_k))\chi^\alpha_k,\quad\mu=1.
\end{eqnarray*}
where $\chi(x)=\frac{e^{-\alpha^2log^2(\frac{x}{1-x})}}{x(1-x)}$, $NI$ is a positive integer.
We can obtain the following error estimate.
 
Let $u$ be the exact solution of Eq.\eqref{1D-FIE-operator}. Let $\tilde{m}=\min\{m,N+1\}$. Suppose that the forcing function $g \in C^m( I)$. Then, there exists a positive integer $N_0$ such that for all $N>N_0$, Eq.\eqref{1D-FIE-nonlinear-system} has a unique solution $\hat{u}^{\log}_N\in \mathcal{P}_{N}^{\log}$ and there exists a positive constant $C$ independent of $N$ such that
$$\|u-\hat{u}^{\log}_N\|_{\chi^\alpha}\leq \tilde{C}N^{\frac16-\frac{\tilde{m}}{2}}\left(\sqrt{\alpha}+N^{-\frac12}\max\{1,\frac{1}{\alpha}\}+\sqrt{2}\alpha\max\{1,\frac{1}{\alpha}\}+N^{-\frac16}\right)\left\|\widehat{\partial}_x^{\tilde{m}}u\right\|_{\chi^{\alpha+\tilde{m}}}.
$$
 \end{remark}
\begin{remark}
The convergence rate of the MHFs-spectral collocation method is closely tied to the regularity of the transformation applied to the exact solution through an appropriate change of variables. The MHFs-spectral collocation method can be viewed as an approach that computes the numerical solution within a transformed polynomial space. This relationship will be further explored in the next sections.
\end{remark}

\section{\bf{Smoothing transformation method for weakly singular Fredholm-Hammerstein integral equations}}
The previous section showed that the convergence rate of the MHFs-spectral collocation method is closely related to the regularity of the transformation of the exact solution by the variable change \eqref{mapping}. Naturally, it occurred to us to consider the smoothing transformation method. We first regularize the solution of Eq. \eqref{1D-FIE-operator} by introducing the variable transformation so that the singularities of the derivatives of the solution will be milder or disappear altogether. Afterward, we solve the transformed equation using the Hermite spectral collocation method and discuss the convergence rate of the obtained approximations. The combination of MHFs-smoothing transformation and Hermite collocation method will be proposed in this section and can be seen as an equivalent method to the MHFs-spectral collocation method . 
\subsection{The numerical scheme}
Introduce the inverse function of \eqref{mapping}:
\begin{equation}\label{tran-fun}
\gamma(x)=\frac{e^{\frac{x}{\alpha}}}{1+e^{\frac{x}{\alpha}}}, \quad x\in\mathbb{R}.
\end{equation}
Introduce in the two-dimensional weakly singular Fredholm-Hammerstein integral equation \eqref{2D-FIE} the change of variables 
$$s=\gamma(\hat{s}),\ t=\gamma(\hat{t}),\ x=\gamma(\hat{x}),\ y=\gamma(\hat{y}),\quad \hat{s},\hat{t},\hat{x},\hat{y}\in\mathbb{R},$$
 and let $$z(\hat{x},\hat{y})=u(\gamma(\hat{x}),\gamma(\hat{y})),\ f(\hat{x},\hat{y})=g(\gamma(\hat{x}),\gamma(\hat{y})),$$ we have
\begin{equation}\label{tran-FIE}
\lambda z(\hat{x},\hat{y})=f(\hat{x},\hat{y})+\int_{\widehat\Omega}\Xi(\hat{s},\hat{t},\hat{x},\hat{y})\hat{\Psi}(\hat{s},\hat{t},z(\hat{s},\hat{t}))d\hat{s}d\hat{t},\quad(\hat{x},\hat{y})\in\widehat\Omega:=\mathbb{R}^2,
\end{equation}
where $$\Xi(\hat{s},\hat{t},\hat{x},\hat{y})=\Theta(\gamma(\hat{s}),\gamma(\hat{t}),\gamma(\hat{x}),\gamma(\hat{y}))$$ and $$\hat{\Psi}(\hat{s},\hat{t},z(\hat{x},\hat{y}))=\psi(\gamma(\hat{s}),\gamma(\hat{t}),z(\hat{x},\hat{y}))\frac{1}{\alpha^2}\frac{e^{\hat{s}/\alpha}}{(1+e^{\hat{s}/\alpha})^2}\frac{e^{\hat{t}/\alpha}}{(1+e^{\hat{t}/\alpha})^2}.$$
The operator form of \eqref{tran-FIE} is as follows
\begin{equation}\label{tran-FIE-operator}
\lambda z(\hat{x},\hat{y})=f(\hat{x},\hat{y})+(\widehat{\boldsymbol{\mathcal{K}}}\widehat{\boldsymbol{\mathcal{R}}}z)(\hat{x},\hat{y}),\quad (\hat{x},\hat{y})\in\widehat\Omega,
\end{equation}
with $$\widehat{\boldsymbol{\mathcal{K}}}z=\int_{\widehat\Omega}\Xi(\hat{s},\hat{t},\hat{x},\hat{y})z(\hat{s},\hat{t})d\hat{s}d\hat{t}$$ and $$\widehat{\boldsymbol{\mathcal{R}}}z=\hat{\Psi}(\hat{s},\hat{t},z(\hat{x},\hat{y})).$$
Then the Hermite spectral collocation method will be used to solve \eqref{tran-FIE} numerically. Let $P_N(\mathbb{R})$ denotes the set of polynomials defined on $\mathbb{R}$ with degree less than $N$. ${\boldsymbol{P_N}}(\widehat\Omega):=P_N\circ P_N(\widehat\Omega)$ denotes the two-dimensional polynomial. The MHFs-smoothing transformation and Hermite collocation method is to find $z_N\in{\boldsymbol{P_N}}(\widehat\Omega)$ such that, for $i,j=0,\cdots N$
\begin{equation}\label{tran-FIE-nonlinear-system}
\lambda  z_N\left(\hat{x}_i,\hat{y}_j\right)=f\left(\hat{x}_i,\hat{y}_j\right)+(\widehat{\boldsymbol{\mathcal{K}}}\widehat{\boldsymbol{\mathcal{R}}}z_N)\left(\hat{x}_i,\hat{y}_j\right),
\end{equation}
where $\left(\hat{x}_i,\hat{y}_j\right)$ are the Hermite-Gauss nodes.
$$ z_N(\hat{x},\hat{y})=\sum\limits_{i=0}^N\sum\limits_{j=0}^NZ_{ij}l_i(\hat{x})l_j(\hat{y}),$$
where $l_i,\ l_j$ are the Lagrange basis function.

And $(\widehat{\boldsymbol{\mathcal{K}}}\widehat{\boldsymbol{\mathcal{R}}})z_N$ can be approximated by Hermite-Gauss quadrature
\begin{eqnarray*}
(\boldsymbol{\widehat{\mathcal{K}}\widehat{\mathcal{R}}})_Nz_{N}=\sum\limits_{k=0}^{NI}\sum\limits_{l=0}^{NI}\Xi(\hat{s},\hat{t},\hat{x},\hat{y})e^{\hat{s}_k^2}e^{\hat{t}_l^2}\psi(\hat{s}_k,\hat{t}_l,z(\hat{s}_k,\hat{t}_l))w_kw_l,
\end{eqnarray*}
 where $\left(\hat{s}_k,\hat{t}_l\right)$ and $w_kw_l$ are the Hermite-Gauss quadrature nodes and weights.
 
 Thus, the discretized MHFs-smoothing transformation and Hermite collocation method is to find $\hat{z}_N=\sum\limits_{i=0}^N\sum\limits_{j=0}^N\hat{Z}_{ij}l_i(\hat{x})l_j(\hat{y})\in{\boldsymbol{P_N}}(\widehat\Omega)$ such that, for $i,j=0,\cdots N$
 \begin{equation}\label{dis-tran-FIE-nonlinear-system}
\lambda  \hat{z}_N\left(\hat{x}_i,\hat{y}_j\right)=f\left(\hat{x}_i,\hat{y}_j\right)+((\widehat{\boldsymbol{\mathcal{K}}}\widehat{\boldsymbol{\mathcal{R}}})_N\hat{z}_N)\left(\hat{x}_i,\hat{y}_j\right).
\end{equation}
\begin{remark}
For the nonlinear integral equation \eqref{tran-FIE},  other fast numerical methods can be further considered in the future. Specifically,  through quasilinearization, we first obtain a linear equation. To approximate the equation,  we truncate the domain and approximate the solution as a step function, using a modified midpoint collocation on a uniform grid. Then the resulting matrix problem is solved using a multigrid method, with an operation count shown to be $\mathcal{O}(NlogN)$ \cite  {10.1007/BF02679436}.
\end{remark}
 \subsection{Error estimate}
In this subsection, we will  provide the  optimal order of global convergence for the MHFs-smoothing transformation and Hermite collocation method in  the one-dimensional case. First we present the nonlinear system of the 1D smoothing transformation method, namely, find $\hat{z}_N\in P_N(\mathbb{R})$ such that, for $i=0,\cdots N$
\begin{equation}\label{1D-tran-FIE-nonlinear-system}
\lambda  \hat{z}_N\left(\hat{x}_i\right)=f\left(\hat{x}_i\right)+((\widehat{\mathcal{K}}\widehat{\mathcal{R}})_N\hat{z}_N)\left(\hat{x}_i\right),
\end{equation}
where $\{\hat{x}_i\}_{i=0}^{N}$ are the Hermite-Gauss nodes.
$$ \hat{z}_N(\hat{x})=\sum\limits_{i=0}^N\hat{Z}_{i}l_i(\hat{x}),$$
where $l_i$ are the Lagrange basis functions. And $(\widehat{\mathcal{K}}\widehat{\mathcal{R}})_N$ is the Hermite-Gauss quadrature operator defined as
\begin{equation}\label{Hermite-Gauss quadrature}
(\widehat{\mathcal{K}}\widehat{\mathcal{R}})_N\hat{z}_N=\sum\limits_{k=0}^{NI}\xi(\hat{s},\hat{x})e^{\hat{s}_k^2}\psi(\hat{s}_k,\hat{x},\hat{z}_N(\hat{s}_k)){w}_k,
\end{equation}
where $\xi(\hat{s},\hat{x})=\theta(\gamma(s),\gamma(t))$ and $\{{w}_k\}_{k=0}^{NI}$ are the Hermite-Gauss weights. 

%

 \begin{thm}\label{H-error-thm}
Let $z$ be the exact solution of the one-dimensional case of Eq.\eqref{tran-FIE}. Suppose that the forcing function $f \in C^m(\mathbb{R})$. Then, there exists a positive integer $N_0$ such that for all $N>N_0$, Eq. \eqref{1D-tran-FIE-nonlinear-system} has a unique solution $\hat{z}_N\in \mathcal{P}_{N}$ and there exists a positive constant $C$ independent of $N$ such that
$$\|z-\hat{z}_N\|_{w}\leq CN^{\frac{1}{6}-\frac{m}{2}}\left(\left\|\partial_x^m z\right\|_\omega+C N^{-\frac16}\left\|z^{(r)} e^{\epsilon x^2} w\right\|_{L_1}\right).$$
\end{thm}

The proof is postponed to Appendix C.

\begin{remark}
From Remark \ref{1D-MHFs-error-thm} and Theorem \ref{H-error-thm}, we observe that for $N+1>m$, the convergence orders of both types of numerical methods for one-dimensional weakly singular Fredholm-Hammerstein integral equations are $\mathcal{O}(N^{\frac16-\frac{m}{2}})$. The theoretical analysis reveals the identical convergence order for the MHFs-spectral collocation method and the MHFs-smoothing transformation and Hermite collocation method. In the next section, we will illustrate the similarity of the convergence order through numerical experiments.
\end{remark}
\section{\bf{Numerical experiments}}
First, we take examples of the one-dimensional linear weakly singular Fredholm integral equations (FIEs).
\begin{example}
We consider the logarithmic type FIE
$$\lambda u_1(x)-\int_0^1\log|x-s|u_1(s)ds=f_1(x),\quad x\in I$$
and the algebraic type FIE
$$\lambda u_2(x)-\int_0^1|x-s|^{-\mu}u_2(s)ds=f_2(x),\quad x\in I$$
with $\lambda=10$ and $\mu=0.5$. Let $u_1=\log(x)\log(1-x)$ and $u_2=x^{\frac12}(1-x)^{\frac12}$ which are weakly singular at 0 and 1. $f_1$ and $f_2$ are confirmed according to $u_1$ and $u_2$.
\end{example}
We use the MHFs-spectral collocation method to solve the above equations numerically and obtain the exponential convergence results of the numerical solutions in Fig.\ref{MHFs-1D}.  Similarly, the MHFs-smoothing transformation and Hermite collocation method are employed to solve the equations numerically, yielding exponential convergence results presented in Fig.\ref{smooth-trans}. From the numerical results, we observe that these two methods exhibit comparable efficiency.
\begin{figure}[htbp]
	\centering
	\begin{minipage}[t]{0.48\textwidth}
		\centering
		\includegraphics[width=7cm]{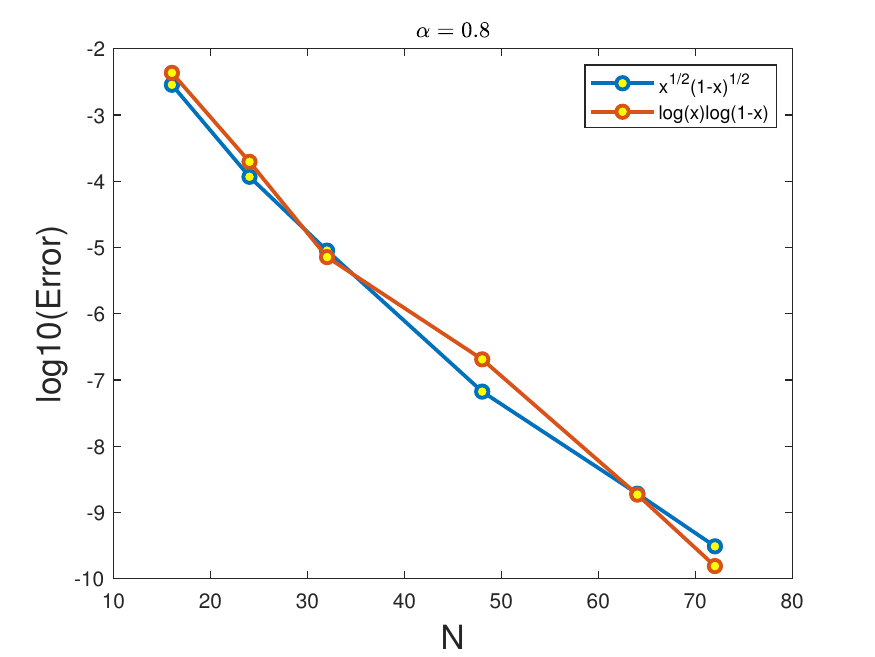}
		\caption{$L^\infty$ norm errors of MHFs-spectral collocation method.}\label{MHFs-1D}
	\end{minipage}
	\begin{minipage}[t]{0.48\textwidth}
		\centering
		\includegraphics[width=7cm]{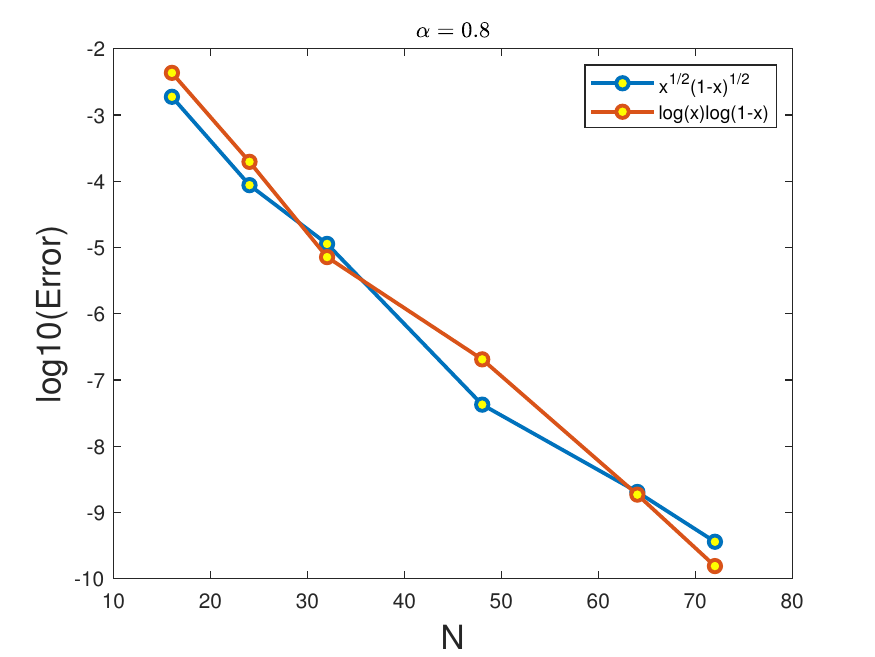}
		\caption{$L^\infty$ norm errors of smoothing transformation-Hermite collocation method.}\label{smooth-trans}
	\end{minipage}
\end{figure}
\begin{example}
Consider the weakly singular Fredholm integral equation with a non-constructive exact solution
$$
u(x)=\sqrt{x}-\frac{\pi}{2}+\int_0^1(1-t)^{-\frac12}u(t) dt,\quad  0 \leq x \leq 1,
$$
with $u(x)=\sqrt{x}$.
\end{example}
The MHFs-spectral collocation method is used to solve the equation numerically and obtain the exponential convergence result of the numerical solution in Fig.\ref{MHFs-1D-2}. 
\begin{figure}[htbp]
	\centering
	\begin{minipage}[t]{0.48\textwidth}
		\centering
\includegraphics[width=7cm]{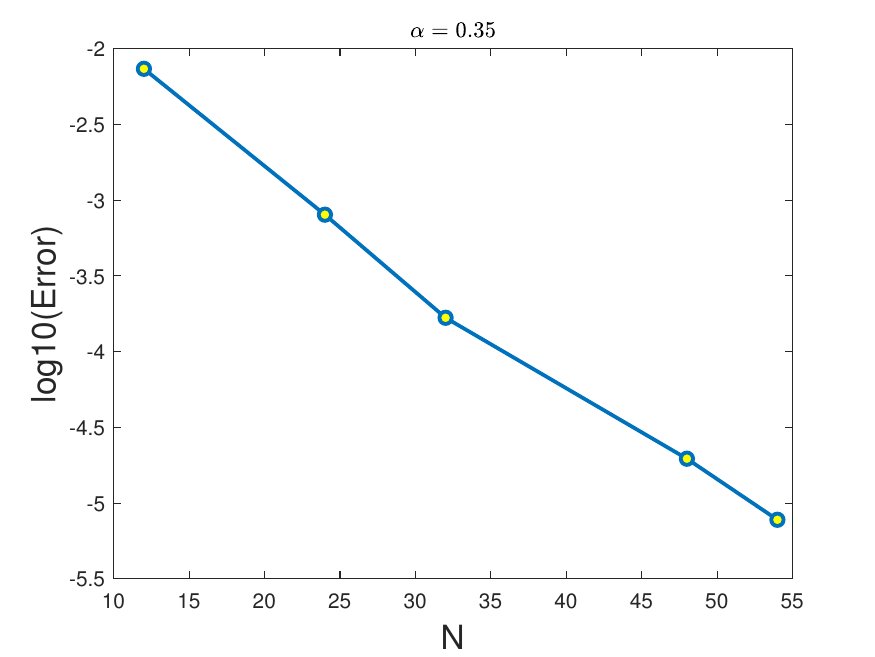}
\caption{$L^\infty$ norm errors of MHFs-spectral collocation method.}\label{MHFs-1D-2}
\end{minipage}
\end{figure}

Next, we  take the two-dimensional weakly singular FHIEs for example.
\begin{example}
We consider the logarithmic type FHIE
\begin{equation}\label{2D-FHIE-log}
\lambda u_1(x,y)=g_1(x,y)+\int_\Omega \log |x-s|\log |y-t|k(s,t,x,y)\psi(s,t,u_1(s,t))dsdt,\ \ (s,t)\in \Omega,
\end{equation}
and the algebraic type FHIE
\begin{equation}\label{2D-FHIE-algeb}
\lambda u_2(x,y)=g_2(x,y)+\int_\Omega |x-s|^{-\mu_1}|y-t|^{-\mu_2}k(s,t,x,y)\psi(s,t,u_2(s,t))dsdt,\ \ (s,t)\in \Omega,
\end{equation}
with $\lambda=10$, $\psi(s,t,u(s,t))=u^2(s,t)$ and $\mu_1=\mu_2=0.5$. Let $u_1=\log(x)\log(1-x)+\log(y)\log(1-y)$ and $u_2=x^{\frac12}(1-x)^{\frac12}y^{\frac12}(1-y)^{\frac12}$ which are weakly singular at the four vertices of $\Omega$. $g_1$ and $g_2$ are confirmed according to $u_1$ and $u_2$.
\end{example}
\begin{figure}[htbp]
	\centering
	\begin{minipage}[t]{0.48\textwidth}
		\centering
		\includegraphics[width=7cm]{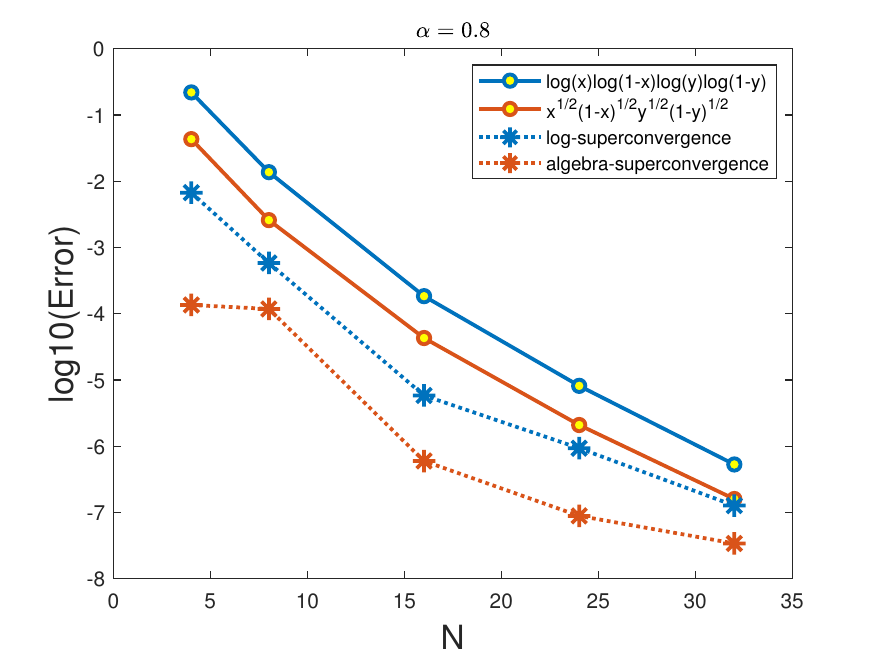}
		\caption{$L^\infty$ norm errors and superconvergence errors of MHFs-spectral collocation method in 2D case.}\label{MHFs-2D}
	\end{minipage}
	\begin{minipage}[t]{0.48\textwidth}
		\centering
		\includegraphics[width=7cm]{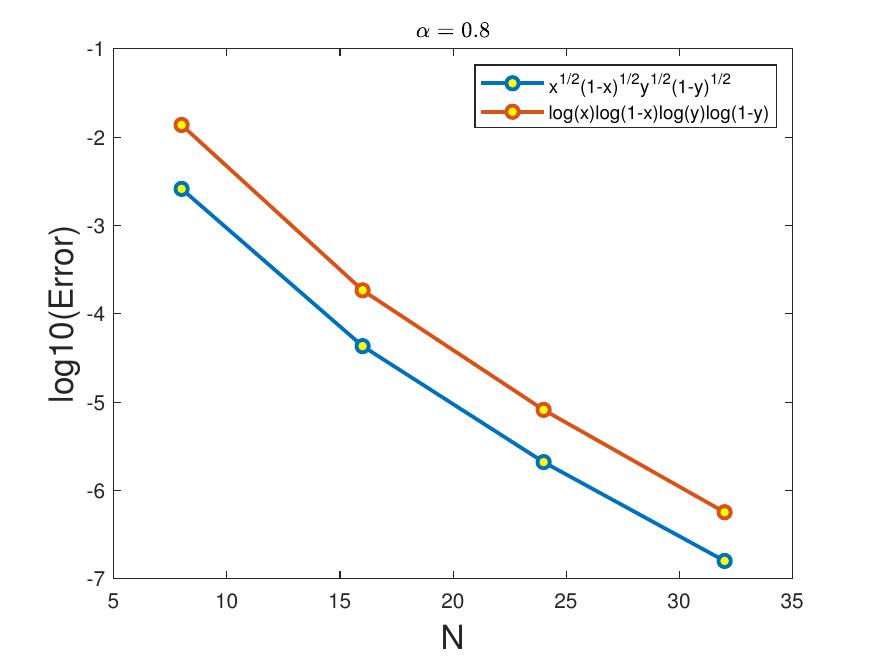}
		\caption{$L^\infty$ norm errors of smoothing transformation-Hermite collocation method in 2D case.}\label{smooth-trans-2D}
	\end{minipage}
\end{figure}
We utilize  the MHFs-spectral collocation method and the quasi-Newton method to numerically solve the two-dimensional FHIE \eqref{2D-FHIE-log} and \eqref{2D-FHIE-algeb}, obtaining exponential convergence results for the numerical solutions  presented in Fig.\ref{MHFs-2D}. Similarly, we employ the MHFs-smoothing transformation technique to solve \eqref{2D-FHIE-log} and \eqref{2D-FHIE-algeb}, yielding exponential convergence results shown in Fig.\ref{smooth-trans-2D}. The identical convergence order  of these two numerical schemes, as predicted by theoretical results, is validated by the numerical experiments.

We evaluate the superconvergence of the MHFs-spectral collocation method at the collocation points. Fig. \ref{MHFs-2D} illustrates that the errors of the numerical solutions at collocation points are significantly reduced, indicating the occurrence of superconvergence with the MHFs-spectral collocation method. 

\section{\bf{Conclusion}}
In this paper, we propose an orthogonal function,  the mapped Hermite function, which is constructed by applying a mapping to Hermite polynomials. Under this mapped function, we propose two kinds of spectral collocation methods to solve a class of two-endpoint weakly singular problems, specifically weakly singular Fredholm-Hammerstein integral equations  that have low regularity at the boundaries of the domain. The first method is the MHFs-spectral collocation method, where the basis functions are constructed based on the mapped Hermite functions. These functions are proven to be suitable for approximating functions with two-endpoint weak singularities, and rigorous error analysis demonstrates the spectral accuracy of the numerical scheme. The second method is the MHFs-smoothing transformation and Hermite collocation method. By utilizing a mapped function on the weakly singular equation, the original equation transforms into a smooth equation defined in $\mathbb{R}$. The Hermite collocation method is then applied to the transformed equation, and the error analysis is rigorously derived. Both methods stem from a mapped function and achieve a similar exponential convergence order. Numerical experiments are presented to validate the error analysis.
\section*{\bf{Acknowledgements}}
This work is supported by NSAF grant in NSFC (No. U2230402) and Postdoctoral Fellowship Program of CPSF (No. GZC20230214).

\section*{\bf{Appendix A}}
The properties of Hermite polynomials:
\begin{enumerate}
  \item Three-term recurrence relation:
\begin{eqnarray}\label{H-three-recur}
\begin{cases}
  H_{n+1}(x)=2 x H_n(x)-2 n H_{n-1}(x), n \geq 1, \\
  H_0(x)=1,\quad\quad H_1(x)=x.
\end{cases}
\end{eqnarray}

  \item Derivative relation:
 \begin{eqnarray}\label{H-deriv-rela}
H_n^{\prime}(x)=\lambda_n H_{n-1}(x), \quad n \geq 1
\end{eqnarray}
and
 \begin{eqnarray}\label{H-deriv-rela2}
H_n^{\prime}(x)=2 x H_n(x)-H_{n+1}(x), \quad n \geq 0 .
\end{eqnarray}
  \item Orthogonality:
 \begin{eqnarray}\label{H-orthogo}
\int_{-\infty}^{+\infty} H_m(x) H_n(x) \omega(x) d x=\gamma_n \delta_{m n}, \quad \gamma_n=\sqrt{\pi} 2^n n !,
\end{eqnarray}
with $\omega(x)=e^{-x^2}$.
  \item Sturm-Liouville problem:
 \begin{eqnarray}\label{H-Sturm}
H_n^{\prime \prime}(x)-2 x H_n^{\prime}(x)+\lambda_n H_n(x)=0 .
\end{eqnarray}
  \item Gauss quadrature:
  $$
\int_{-\infty}^{+\infty} p(x) e^{-x^2} d x=\sum_{j=0}^N p\left(x_j\right) \omega_j, \quad \forall p \in P_{2 N+1},
$$
where $P_{ N}$ is the polynomial of degree $\leq N$.
\end{enumerate}

\section*{\bf{Appendix B: Proof of Theorem \ref{2D-MHFs-error-thm}}}
We first present some necessary lemmas for its proof.
\begin{lmm}\label{2D-Sloan-compact}
Let the kernel $\boldsymbol{\mathcal{KR}}$ be given by
$$(\boldsymbol{\mathcal{K}}\boldsymbol{\mathcal{R}})u=\int_\Omega\Theta(s,t,x,y)\psi(s,t,u(s,t))dsdt\approx(\boldsymbol{\mathcal{K}}\boldsymbol{\mathcal{R}})_Nu,$$
where $\psi$ are continuous. The MHFs-Gauss quadrature operator $(\boldsymbol{\mathcal{K}}\boldsymbol{\mathcal{R}})_N$ is defined in \eqref{2D-quadra}. Denote 
\begin{equation}\label{2D-quadra-W}
W_{kl}(x,y)=\frac{\Theta(s_k^{(\alpha_1)},t_l^{(\alpha_2)},x,y)}{\chi^{\alpha_1}(s_k^{(\alpha_1)})\chi^{\alpha_2}(t_l^{(\alpha_2)})}\chi^{\alpha_1}_k\chi^{\alpha_2}_l
\end{equation}
 and $W_{kl}$ satisfies
$$W_{kl}\in E,$$
and
$$\lim\limits_{x'\rightarrow x}\lim\limits_{y'\rightarrow y}\|W_{kl}(x',y')-W_{kl}(x,y)\|_E=0$$
for all $x\in I$, and where $E$ is a Banach space continuously imbedded in $L^1( I)$. Further, the approximation to the integral term $(\boldsymbol{\mathcal{K}}\boldsymbol{\mathcal{R}})_N$ has the property
\begin{equation}\label{2D-quadra-lim-equ}
\lim\limits_{N\rightarrow\infty}\sum\limits_{k=0}^{N}\sum\limits_{l=0}^{N}\frac{\Theta(s_k^{(\alpha_1)},t_l^{(\alpha_2)},x,y)}{\chi^{\alpha_1}(s_k^{(\alpha_1)})\chi^{\alpha_2}(t_l^{(\alpha_2)})}\psi(s_k^{(\alpha_1)},t_l^{(\alpha_2)},u(s_k^{(\alpha_1)},t_l^{(\alpha_2)}))\chi^{\alpha_1}_k\chi^{\alpha_2}_l=\int_\Omega\Theta(s,t,x,y)\psi(s,t,u(s,t))dsdt
\end{equation}
for all $\psi\in C( \Omega)$ and all $\theta\in E$. Then the sequence $\{(\boldsymbol{\mathcal{K}}\boldsymbol{\mathcal{R}})_N\}$ is a collectively compact set of operators on $C(\Omega)$ with the property that $\|(\boldsymbol{\mathcal{K}}\boldsymbol{\mathcal{R}})_N\psi-\boldsymbol{\mathcal{K}}\boldsymbol{\mathcal{R}}\psi\|_\infty\rightarrow0$ as $n\rightarrow\infty$ for all $\psi\in C( \Omega)$.
\end{lmm}
\begin{proof}
This result can be extended by Theorem 2 in \cite{sloan1981analysis}. We omit here and leave it to the reader. 
\end{proof}
For $r \geq 0$ and $\kappa \in(0,1), C^{r, \kappa}( I)$ will denote the space of functions whose $r$-th derivatives are H\"older continuous with exponent $\kappa$, endowed with the usual norm:
$$
\|v\|_{r, \kappa}=\max _{0 \leq i \leq r} \max _{x \in I}\left|\partial_x^i v(x)\right|+\max _{0 \leq i \leq r} \sup _{x, y \in I, x \neq y} \frac{\left|\partial_x^i v(x)-\partial_x^i v(y)\right|}{|x-y|^\kappa} .
$$

When $\kappa=0, C^{r, 0}( I)$ denotes the space of functions with $r$ continuous derivatives on $ I$, which is also commonly denoted by $C^r( I)$, endowed with the norm $\|\cdot\|_r$. Next we will prove that one-dimensional integral operator $\mathcal{K}$ defined as
$$\mathcal{K}v=\int_0^1\theta(s,t)\psi(u(s),s)ds$$ 
is a compact operator from $C( I)$ to $C^{0,\kappa}( I)$ for any $0 <\kappa< 1 - \mu$.
\begin{lmm}\label{compact-lmm-integral}
$\forall v \in C( I), \theta \in C( I \times  I)$, and $\theta( s,\cdot) \in C^{0, \kappa}( I)$ with $0<\kappa<1-\mu$, we have
\begin{equation}\label{compact-lmm}
\frac{|(\mathcal{K} v)(x)-(\mathcal{K} v)(y)|}{|x-y|^\kappa} \leq c \max _{x \in I}|v(x)|, \quad \forall x, y \in  I, x \neq y .
\end{equation}

This implies that
\begin{equation}\label{compact-lmm1}
\|\mathcal{K} v\|_{0, \kappa} \leq c\|v\|_{\infty}, \quad 0<\kappa<1-\mu .
\end{equation}
\end{lmm}
\begin{proof}  1) For $0<\mu<1$, $\forall v\in C( I)$, we have 
$$\mathcal{K} v=\int_0^x\theta(s,x)u(s)ds+\int_x^1\theta(s,x)u(s)ds:=\mathcal{K}_1 v+\mathcal{K}_2 v.$$
For the kernel $\mathcal{K}_1$, we can immediately obtain the compactness by Lemma 2.12 in \cite{hou2019muntz}. For the kernel $\mathcal{K}_2$, we substitute the $\int_0^x$ with $\int_x^1$  in the same lemma from \cite{hou2019muntz} and employ the similar proof, thereby confirming the compactness of  $\mathcal{K}_2$. Hence, $\mathcal{K} v$ is compact when $0<\mu<1$.\\
2) For $\mu=1$, $\forall v\in C( I)$, we have 
$$\mathcal{K} v=\int_0^x\theta(s,x)u(s)ds+\int_x^1\theta(s,x)u(s)ds:=\mathcal{K}_1 v+\mathcal{K}_2 v.$$

First, we proof the compactness of $\mathcal{K}_1$.
Without loss of generality, assume $0 \leq y<x \leq 1$, We have
$$
\begin{aligned}
\frac{|(\mathcal{K}_1 v)(x)-(\mathcal{K}_1 v)(y)|}{|x-y|^\kappa} & =(x-y)^{-\kappa} \left| \int_0^y\log|y-s| k( s,y) v(s) d s -\int_0^x\log|x-s| k(s,x ) v(s) d s \right|\\
& \leq K_1+K_2,
\end{aligned}
$$
where
$$
\begin{aligned}
& K_1=(x-y)^{-\kappa} \int_0^y\Big|\log|y-s| k(s, y )-\log|x-s| k(s, x )\Big||v(s)| d s, \\
& K_2=(x-y)^{-\kappa} \int_y^x\log|x-s||k(s, x )||v(s)| d s .
\end{aligned}
$$

For $K_1$, by the triangle inequality, we have
$$
K_1 \leq K^{(1)}+K^{(2)},
$$
where
$$
\begin{aligned}
& K^{(1)}=(x-y)^{-\kappa} \int_0^y\Big|\log|y-s|-\log|x-s|\Big||k(s, y )||v(s)| d s, \\
& K^{(2)}=(x-y)^{-\kappa} \int_0^y\log|x-s||k(s, y )-k(s, x ) \| v(s)| d s .
\end{aligned}
$$
Then, we will estimate $K^{(1)}$ and $K^{(2)}$, respectively.
$$
\begin{aligned}
K^{(1)} \leq & c\|v\|_{\infty}(x-y)^{-\kappa}\left[\int_0^y\log|y-s| d s-\int_0^x\log|x-s| d s+\int_y^x\log|x-s| d s\right] \\
\leq & c\|v\|_{\infty}(x-y)^{-\kappa}\left[\int_0^1\log|y-ys|yds-\int_0^1\log|x-xs|xds\right. \\
& \left.+\int_0^1\log|x-(x-y)s-y|(x-y)ds|\right] \\
= & c\|v\|_{\infty}(x-y)^{-\kappa}\left[y\log y-x\log x+(x-y)\log(x-y)+y\int_0^1\log(1-s)ds\right. \\
& \left.-x\int_0^1\log(1-s)ds+(x-y)\int_0^1\log(1-s)ds|\right] \\
& \leq c \left[y\log y-x\log x+(x-y)\log(x-y)\right](x-y)^{-\kappa}\|v\|_{\infty} \leq c\|v\|_{\infty}, \\
K^{(2)} & =\int_0^y\log(x-s) \frac{|k(s, y )-k(s, x )|}{(x-y)^\kappa}|v(s)| d s \\
& \leq c \max _{s \in I}\|k(s,\cdot)\|_{0, \kappa}\|v\|_{\infty} \int_0^x\log(x-s) d s \\
&  \leq c\|v\|_{\infty},
\end{aligned}
$$

For $K_2$, we have
$$
K_2 \leq c\|v\|_{\infty}(x-y)^{-\kappa} \int_y^x\log(x-s) d s \leq c\|v\|_{\infty} .
$$
Thus, 
\begin{equation}\label{kernel-K1}
\frac{|(\mathcal{K}_1 v)(x)-(\mathcal{K}_1 v)(y)|}{|x-y|^\kappa} \leq c \max _{x \in I}|v(x)|, \quad \forall x, y \in  I, x \neq y,
\end{equation}
and
$$\|\mathcal{K}_1 v\|_{0, \kappa} \leq c\|v\|_{\infty}, \quad 0<\kappa<1-\mu .$$

Next, we establish the compactness of $\mathcal{K}_2$ similarly to the estimate with $\mathcal{K}_1$.
$$
\begin{aligned}
\frac{|(\mathcal{K}_2 v)(x)-(\mathcal{K}_2 v)(y)|}{|x-y|^\kappa} & =(x-y)^{-\kappa} \left| \int_y^1\log|y-s| k(s,y ) v(s) d s -\int_x^1\log|x-s| k(s,x ) v(s) d s \right|\\
& \leq K^\prime_1+K^\prime_2,
\end{aligned}
$$
where
$$
\begin{aligned}
& K^\prime_1=(x-y)^{-\kappa} \int_x^1\Big|\log|y-s| k(s, y )-\log|x-s| k(s, x )\Big||v(s)| d s, \\
& K^\prime_2=(x-y)^{-\kappa} \int_y^x\log|x-s||k(s, x )||v(s)| d s .
\end{aligned}
$$
By analogous deducing, we can obtain that
$$K^\prime_1\leq c\|v\|_{\infty},\quad K^\prime_2\leq c\|v\|_{\infty}.$$
Thus, we have
\begin{equation}\label{kernel-K2} 
\frac{|(\mathcal{K}_2 v)(x)-(\mathcal{K}_2 v)(y)|}{|x-y|^\kappa} \leq c \max _{x \in I}|v(x)|, \quad \forall x, y \in  I,\ x \neq y.
\end{equation}
and
$$\|\mathcal{K}_2 v\|_{0, \kappa} \leq c\|v\|_{\infty}, \quad 0<\kappa<1-\mu .$$

Consolidating all these estimates thus leads to \eqref{compact-lmm} and \eqref{compact-lmm1}.
\end{proof}
\begin{remark}\label{compact-remark}
From the properties of $\psi$ in A3-A5, the operator $\mathcal{R}$ is compact from $C( I)$ to itself. Combining the compactness of $\mathcal{K}$ in Lemma \ref{compact-lmm-integral}, we obtain that $\mathcal{K}\mathcal{R}$ is a compact operator from $C( I)$ to $C^{0,\kappa}( I)$ for any $0 <\kappa< 1 - \mu$.
\end{remark}
We extend Lemma \ref{compact-lmm-integral} and Remark \ref{compact-remark} to the two-dimensional case and obtain the compactness of the integral operator $\boldsymbol{\mathcal{K}\mathcal{R}}$. First, we present the definition of the multidimensional H\"older space.
\begin{defn}
 For every bounded function $v(\mathbf{x})$, there exists a constant $C$ independent of $v$ such that
$$
\sup _N\left\|\sum_{\|\mathbf{j}\| \leq N} v\left(\tilde{\mathbf{x}}_{\mathbf{j}}\right) F_{\mathbf{j}}(\mathbf{x})\right\|_{L_\omega^2(\Omega)} \leq C \max _{\mathbf{x} \in \bar{\Omega}}|v(\mathbf{x})| .
$$
For $r \geq 0$ and $\kappa \in(0,1), \mathcal{C}^{r, \kappa}(\bar{\Omega})$ will denote the space of functions whose $r$-th derivatives are H\"older continuous with exponent $\kappa$, endowed with the norm:
$$
\begin{aligned}
&\|v\|_{C^{r, \kappa}(\bar{\Omega})}=\max _{|\alpha| \leq r} \max _{\mathbf{x} \in \bar{\Omega}}\left|\frac{\partial^{|\alpha|} v(\mathbf{x})}{\partial x_1^{\alpha_1} \partial x_2^{\alpha_2} \cdots \partial x_d^{\alpha_d}}\right| \\
&+\max _{|\alpha| \leq r} \sup _{\mathbf{x}^{\prime} \neq \mathbf{x}^{\prime \prime} \in \bar{\Omega}}\left|\frac{\frac{\partial^{|\alpha|} \mid v\left(\mathbf{x}^{\prime}\right)}{\partial x_1^{\alpha_1} \partial x_2^{\alpha_2} \cdots \partial x_d^{\alpha_d}}-\frac{\partial^{|\alpha|} v\left(\mathbf{x}^{\prime \prime}\right)}{\partial x_1^{\alpha_1} \partial x_2^{\alpha_2} \cdots \partial x_d^{\alpha_d}}}{\left[\left(x_1^{\prime}-x_1^{\prime \prime}\right)^2+\left(x_2^{\prime}-x_2^{\prime \prime}\right)^2+\cdots+\left(x_d^{\prime}-x_d^{\prime \prime}\right)^2\right]^{\frac{\kappa}{2}}}\right| .
\end{aligned}
$$
 \end{defn}

\begin{lmm}\label{2D-compact-lmm-integral}
$\forall v \in C( \Omega), \Theta \in C( \Omega \times  \Omega)$, and $\Theta( s,t,\cdot,\cdot) \in C^{0, \kappa}( \Omega)$ with $0<\kappa<1-\mu$, we have
\begin{equation}\label{2D-compact-lmm}
\frac{|(\boldsymbol{\mathcal{K}} v)(x_1,y_1)-(\boldsymbol{\mathcal{K}} v)(x_2,y_2)|}{[(x_1-x_2)^2+(y_1-y_2)^2]^{\frac{\kappa}{2}}} \leq c \max _{(x,y) \in \Omega}|v(x,y)|, \quad \forall (x_i, y_i )\in  \Omega, x_i \neq y_i\ \text{for}\ i=1,2.
\end{equation}
This implies that
\begin{equation}
\|\boldsymbol{\mathcal{K}} v\|_{0, \kappa} \leq c\|v\|_{\infty}, \quad 0<\kappa<1-\mu .
\end{equation}
\end{lmm}
\begin{proof}
This result can be extended by Lemma \ref{compact-lmm-integral}. We omit here and leave it to the reader. 
\end{proof}
\begin{remark}\label{2D-compact-lmm-integral-KR}
From the properties of $\psi$ in A3-A5, the operator $\boldsymbol{\mathcal{R}}$ is compact from $C( \Omega)$ to itself. Combining the compactness of $\boldsymbol{\mathcal{K}}$ in Lemma \ref{compact-lmm-integral}, we obtain that $\boldsymbol{\mathcal{K}\mathcal{R}}$ is a compact operator from $C( \Omega)$ to $C^{0,\kappa}( \Omega)$ for any $0 <\kappa< 1 - \mu$.
\end{remark}

\begin{proofthm1}
First it is easy to verify that $W_{kl}(x,y)$ satisfies
$$\int_0^1|W_{kl}(x,y)\boldsymbol{\chi^\alpha}(x,y)|^pdx<\infty$$
and 
 $$\lim\limits_{(x',y')\rightarrow (x,y)}\|W_{kl}(x',y')-W_{kl}(x,y)\|_{p,\boldsymbol{\chi^\alpha}}=0.$$
 In view of Theorem \ref{2D-quadra-thm}, \eqref{2D-quadra-lim-equ} is true. Then by Lemma \ref{2D-Sloan-compact}, $\{(\boldsymbol{\mathcal{K}\mathcal{R}})_N\}$ is collectively compact. Notice that $\lambda I-(\boldsymbol{\mathcal{K}\mathcal{R}})_N$ is injective from $L^\infty( \Omega)$ to itself. There exists a positive constant $C$ and a positive integer $N_0$ such that for all $N>N_0$, $(\lambda I-(\boldsymbol{\mathcal{KR}})_N)^{-1}$ exists and is uniformly bounded by $C$ in the $L^\infty$ norm. Thus, Eq.\eqref{dis-FIE_nonlinear_system} has a unique solution $$\hat{u}^{\log}_N=(\lambda I-(\boldsymbol{\mathcal{KR}})_N)^{-1}\boldsymbol{\mathcal{I}_N^{\alpha}} g \in\mathcal{P}_{N}^{\log}.$$
 
 In view of
 $$u=\frac{1}{\lambda}\left((\boldsymbol{\mathcal{KR}})u+g\right),$$
 we have
 \begin{equation}\label{error-equation}
 u-\hat{u}^{\log}_N=\frac{1}{\lambda}\left((\boldsymbol{\mathcal{KR}})u-(\boldsymbol{\mathcal{KR}})_N\hat{u}^{\log}_N\right)+g-\boldsymbol{\mathcal{I}_N^{\alpha}} g.
 \end{equation}
 Through a simple derivation, we have
 $$g-\boldsymbol{\mathcal{I}_N^{\alpha}} g=(\lambda I-\boldsymbol{\mathcal{K}\mathcal{R}})^{-1}(u-\boldsymbol{\mathcal{I}_N^{\alpha}}u).$$
 Thus, \eqref{error-equation} can be rewritten as
  $$u-\hat{u}^{\log}_N=\frac{1}{\lambda}\left((\boldsymbol{\mathcal{KR}})u-(\boldsymbol{\mathcal{KR}})_N\hat{u}^{\log}_N\right)+(\lambda I-\boldsymbol{\mathcal{K}\mathcal{R}})^{-1}(u-\boldsymbol{\mathcal{I}_N^{\alpha}}u).$$
  For $(\boldsymbol{\mathcal{KR}})u-(\boldsymbol{\mathcal{KR}})_N\hat{u}^{\log}_N$,
  \begin{eqnarray*}
    (\boldsymbol{\mathcal{KR}})u-(\boldsymbol{\mathcal{KR}})_N\hat{u}^{\log}_N&=&(\boldsymbol{\mathcal{KR}})u-(\boldsymbol{\mathcal{KR}})_Nu+(\boldsymbol{\mathcal{KR}})_Nu-(\boldsymbol{\mathcal{KR}})_N\hat{u}^{\log}_N  \\
    &=& \left((\boldsymbol{\mathcal{KR}})-(\boldsymbol{\mathcal{KR}})_N\right)u+(\boldsymbol{\mathcal{KR}})_N(u-\hat{u}^{\log}_N)
  \end{eqnarray*}
  $$u-\hat{u}^{\log}_N=\frac{1}{\lambda}(I-\frac{1}{\lambda}(\boldsymbol{\mathcal{KR}})_N)^{-1}\left((\boldsymbol{\mathcal{KR}})-(\boldsymbol{\mathcal{KR}})_N\right)u+(\lambda I-\boldsymbol{\mathcal{K}\mathcal{R}})^{-1}(u-\boldsymbol{\mathcal{I}_N^{\alpha}}u).$$
 By triangle inequality, Theorem \ref{2D-interpolation-error}, Theorem \ref{2D-quadra-thm} and compactness of $\boldsymbol{\mathcal{KR}}$ in Lemma \ref{2D-compact-lmm-integral} and Remark \ref{2D-compact-lmm-integral-KR}, we can imply that
\begin{eqnarray*}
\|u-\hat{u}^{\log}_N\|_{\chi^\alpha}&\leq& C N^{\frac13-\tilde{m}}\left[N^{-\frac13}+(C^*+\beta^2 N^{-1})^{1/2}(\alpha^{\frac12}+\beta (N^{-\frac12}+N^{\frac12}))\right]\left\|\widehat{D}^{\tilde{m}}u\right\|_{\boldsymbol{\chi^{\alpha+\tilde{m}}}}\\
&\leq&C N^{\frac56-\tilde{m}}\left\|\widehat{D}^{\tilde{m}}u\right\|_{\boldsymbol{\chi^{\alpha+\tilde{m}}}}.
\end{eqnarray*}
\end{proofthm1}

\section*{\bf{Appendix C: Proof of Theorem \ref{H-error-thm}}}
We first present some necessary lemmas for its proof.

We denote by $A C_{\mathrm{loc}}$ the space of functions which are absolutely continuous on every closed subset of $(-\infty,\infty)$. The error bounds for the Gauss-Hermite quadrature formula \eqref{Hermite-Gauss quadrature} will be introduced.
\begin{lmm}[\cite{mastroianni1994error}, Theorem 2]\label{H-Gauss-quad-error}
Let $p^{(m-1)} \in A C_{\mathrm{loc}}$ and $p^{(m)} e^{\epsilon x^2} w \in L_1$ for some $m \geq 1$ and some $0<\epsilon<1$. Then we have
$$
\left|\int_{-\infty}^{+\infty} p(x) w(x) d x-\sum_{j=0}^N p\left(x_j\right) \omega_j\right| \leq C N^{-m / 2}\left\|p^{(m)} e^{\epsilon x^2} w\right\|_{L_1},
$$
where $C$ is a constant independent of $N$ and $p$.
\end{lmm}
Next we introduce the interpolation error estimates associated with the Hermite-Gauss quadrature results in \eqref{Hermite-Gauss quadrature}.
\begin{lmm}[\cite{shen2011spectral},Theorem 7.17]\label{H-interpolation error}
For $u \in C(\mathbb{R}) \cap H_\omega^m(\mathbb{R})$ with $m \geq 1$, we have
$$
\left\|\partial_x^l\left(\mathcal{I}_N^h u-u\right)\right\|_\omega \lesssim N^{\frac{1}{6}+\frac{l-m}{2}}\left\|\partial_x^m u\right\|_\omega, \quad 0 \leq l \leq m,
$$
and $\forall v\in C(\mathbb{R})$, the interpolation operator $\mathcal{I}_N^h: C(\mathbb{R})\rightarrow P_N(\mathbb{R})$ is defined as
$$\mathcal{I}_N^hv=\sum\limits_{i=0}^{N}v(\hat{x}_i)l_i(\hat{x}).$$

\end{lmm}
Besides, we need to proof the integral operator and the Hermite-Gauss quadrature operator defined in \eqref{Hermite-Gauss quadrature} is collectively compact, respectively. Here, for convenience of the proof, we denote the integral operator as
$$\widehat{\mathcal{K}}(p)=\int_{\mathbb{R}}\xi(\hat{s},\hat{x})\hat{\Psi}(\hat{s},\hat{x},z(\hat{x}))d\hat{s}$$
\begin{lmm}\label{H-integral-compact}
For any function $v(x) \in C(\mathbb{R}), \xi \in C(\mathbb{R}\times\mathbb{R})$, and $\xi( s,\cdot) \in C^{0, \kappa}(\mathbb{R})$ with $0<\kappa<1-\mu$, there exists a positive constant $c$ such that
$$
\frac{\left|(\widehat{\mathcal{K}} v)\left(\gamma(\hat{x})\right)-(\widehat{\mathcal{K}} v)\left(\gamma(\hat{y})\right)\right|}{|\hat{x}-\hat{y}|^\kappa} \leq c \max _{\hat{x} \in \mathbb{R}}|v(\gamma(\hat{x}))|, \quad \forall \hat{x}, \hat{y} \in \mathbb{R}, \hat{x} \neq \hat{y},
$$
where $\gamma(\hat{x})$ is the transformed function defined in \eqref{tran-fun}. Thus
\begin{equation}\label{H-integral-compact3}
\left\|(\widehat{\mathcal{K}} v)\left(\gamma(\hat{x})\right)\right\|_{0, \kappa} \leq c\|v\|_{\infty},\quad \hat{x}\in\mathbb{R}.
\end{equation}
\end{lmm}
\begin{proof}
It follows from \eqref{compact-lmm} that
\begin{equation}\label{H-integral-compact1}
\frac{\left|(\widehat{\mathcal{K}} v)\left(\gamma(x)\right)-(\widehat{\mathcal{K}} v)\left(\gamma(y)\right)\right|}{\left|\gamma(\hat{x})-\gamma(\hat{y})\right|^\kappa} \leq c \max _{\hat{x} \in \mathbb{R}}|v(\gamma(\hat{x}))|, \quad \forall \hat{x}, \hat{y} \in \mathbb{R}, \hat{x} \neq \hat{y} .
\end{equation}

Noticing that, for any $\hat{x}, \hat{y} \in \mathbb{R}, \hat{x} \neq \hat{y}, 0<\lambda \leq 1$, we have
$$
\left|\gamma(\hat{x})-\gamma(\hat{y})\right|^\kappa \sim o\left(|\hat{x}-\hat{y}|^\kappa\right) .
$$

Thus we obtain
\begin{equation}\label{H-integral-compact2}
\frac{\left|(\widehat{\mathcal{K}} v)\left(\gamma(\hat{x})\right)-(\widehat{\mathcal{K}} v)\left(\gamma(\hat{y})\right)\right|}{|\hat{x}-\hat{y}|^\kappa} \leq c \frac{\left|(\widehat{\mathcal{K}} v)\left(\gamma(\hat{x})\right)-(\widehat{\mathcal{K}} v)\left(\gamma(\hat{y})\right)\right|}{\left|\gamma(\hat{x})-\gamma(\hat{y})\right|^\kappa} .
\end{equation}

Combining with \eqref{H-integral-compact1} and \eqref{H-integral-compact2} gives \eqref{H-integral-compact3}. Thus we complete the proof.
\end{proof}
\begin{remark}\label{H-integral-compact-KR}
From the properties of $\psi$ in A3-A5 and by the similar analysis with Lemma \ref{H-integral-compact}, the operator $\mathcal{R}$ is compact from $C( \mathbb{R})$ to itself. Combining the compactness of $\mathcal{K}$ in Lemma \ref{compact-lmm-integral}, we obtain that $\mathcal{K}\mathcal{R}$ is a compact operator from $C( \mathbb{R})$ to $C^{0,\kappa}( \mathbb{R})$ for any $0 <\kappa< 1 - \mu$.
\end{remark}
\begin{lmm}\label{H-kernel-compact}
Let $C_l$ be the space of continuous functions on $\mathbb{R}$ having a limit as $x\rightarrow\pm\infty$. The sequence $\left\{\mathcal{K}_N\right\}$ is a collectively compact set of operators from $C_l$ to itself, with the property that $\left\|\mathcal{K}_N g-\mathcal{K} g\right\|_{\infty} \rightarrow 0$ as $n \rightarrow \infty$ for all $g \in C_l$, if and only if the following conditions hold:
\begin{equation}\label{lmm-cond1}
\lim _{N \rightarrow \infty} \sum_{i=0}^N W_{N i}(t) g\left(s_{N i}\right)=\int_{\mathbb{R}} k (s, t) g(s) d s
\end{equation}
for all $g \in C_l$ and all $t \in\mathbb{R}$;
\begin{equation}\label{lmm-cond2}
\lim _{t^{\prime} \rightarrow t} \sup _N \sum_{i=0}^N\left|W_{N i}\left(t^{\prime}\right)-W_{N i}(t)\right|=0
\end{equation}
for all $t \in\mathbb{R}$; and
\begin{equation}\label{lmm-cond3}
\lim _{t \rightarrow \infty} \sup _{t^{\prime}>t} \sup _N \sum_{i=0}^n\left|W_{N i}\left(t^{\prime}\right)-W_{N i}(t)\right|=0 .
\end{equation}
\end{lmm}
\begin{proof}
The proof process is similar with the Lemma in \cite{sloan1981quadrature}. The details can refer \cite{sloan1981quadrature}.

\end{proof}

\begin{proofthm}
For $0<\mu<1$,
\begin{eqnarray*}
((\widehat{\mathcal{K}}\widehat{\mathcal{R}})_Nz)(\hat{x})&=&
\begin{dcases}\sum\limits_{k=1}^N|\hat{x}-\hat{s}_k|^{-\mu}k(\hat{s}_k,\hat{x})e^{\hat{s}_k^2}\psi(\hat{s}_k,\hat{x},z(\hat{s}_k))\widehat{\chi}^\alpha_k\\
\sum\limits_{k=1}^N\log|\hat{x}-\hat{s}_k|k(\hat{s}_k,\hat{x})e^{\hat{s}_k^2}\psi(\hat{s}_k,\hat{x},z(\hat{s}_k))\widehat{\chi}^\alpha_k
\end{dcases}\\
&:=&\sum\limits_{k=1}^NW_{N k}(\hat{x})z(\hat{s}_k).
\end{eqnarray*}
First it is easy to verify that $W_{N k}(\hat{x})$ satisfies \eqref{lmm-cond1}, \eqref{lmm-cond2} and \eqref{lmm-cond3}. Then by Lemma \ref{H-kernel-compact}, $\{(\widehat{\mathcal{K}}\widehat{\mathcal{R}})_N\}$ is collectively compact. And notice that $\lambda I-(\widehat{\mathcal{K}}\widehat{\mathcal{R}})_N$ is injective from $L^\infty(\mathbb{R})$ to itself.There exist a positive constant $C$ and a positive integer $N_0$ such that for all $N>N_0$, $(\lambda I-(\widehat{\mathcal{K}}\widehat{\mathcal{R}})_N)^{-1}$ exists and is uniformly bounded by $C$ in the $L^\infty(\mathbb{R})$ norm. Thus, Eq. \eqref{1D-tran-FIE-nonlinear-system} has a unique solution $$\hat{z}_N=(\lambda I-(\widehat{\mathcal{K}}\widehat{\mathcal{R}})_N)^{-1}\mathcal{I}_N^h g\in\mathcal{P}_{N}.$$
 We have 
$$z-\hat{z}_N =\frac{1}{\lambda}(I-\frac{1}{\lambda}(\widehat{\mathcal{K}}\widehat{\mathcal{R}})_N)^{-1}\left((\widehat{\mathcal{K}}\widehat{\mathcal{R}})-(\widehat{\mathcal{K}}\widehat{\mathcal{R}})_N\right)z+(\lambda I-\widehat{\mathcal{K}}\widehat{\mathcal{R}})^{-1}(z-\mathcal{I}_N^hz).$$
Thus, by triangle inequality, Lemma \ref{H-interpolation error}, Lemma \ref{H-Gauss-quad-error} and compactness of $\mathcal{\widehat{K}\widehat{R}}$ in Lemma \ref{H-integral-compact} and Remark \ref{H-integral-compact-KR}, we can obtain
$$\|z-\hat{z}_N\|_{w}\leq CN^{\frac{1}{6}+\frac{-m}{2}}\left(\left\|\partial_x^m z\right\|_\omega+C N^{-\frac16}\left\|z^{(r)} e^{\epsilon x^2} w\right\|_{L_1}\right).$$
\end{proofthm}

\bibliographystyle{plain}
\bibliography{Hermite-trans-bib}

\end{document}